\documentclass[11pt,fleqn]{article}

\usepackage{amssymb, amsthm, amsmath}

\numberwithin{equation}{section}

\theoremstyle{plain}
\newtheorem{theorem}{Theorem}[section]
\newtheorem{lemma}{Lemma}[section]

\theoremstyle{definition}

\theoremstyle{remark}

\title{Global $L_{p}$-$L_{q}$ estimates for solutions to the third initial-boundary value problem for the heat equation in a bounded domain}

\author{Ry\^{o}hei Kakizawa\thanks{Graduate School of Mathematical Sciences, The University of Tokyo, 3-8-1 Komaba Meguro-ku Tokyo 153-8914, Japan (\textit{E-mail address:} kakizawa@ms.u-tokyo.ac.jp)}}

\date{}

\begin{document}

\maketitle

\begin{abstract}
We discuss the unique solvability of the third initial-boundary value problem for the heat equation in a bounded domain.
This problem has uniquely a time-global solution in the anisotropic Sobolev space $W^{2,1}_{p,q}$ for any $1<p<\infty$, $1<q<\infty$.
Moreover, exponentially weighted $L_{p}$-$L_{q}$ estimates for time-global solutions can be established.
We prove the above properties by $L_{p}$ estimates for steady solutions to the heat equation, the theory of analytic semigroups on Banach spaces and the operator-valued Fourier multiplier theorem on UMD spaces.
\end{abstract}


\section{Introduction}
Let $\Omega$ be a bounded domain in $\mathbb{R}^{n}$ $(n \in \mathbb{Z}, \ n\geq 2)$ with its $C^{2}$-boundary $\partial\Omega$.
We consider the following initial-boundary value problem for the heat equation:
\begin{equation}
\begin{split}
&\partial_{t}u-\mathrm{div}(\kappa\nabla u)=f & \mathrm{in} \ \Omega\times(0,T), \\
&u|_{t=0}=u_{0} & \mathrm{in} \ \Omega, \\
&\kappa\partial_{\nu}u+\kappa_{s}u|_{\partial\Omega}=g & \mathrm{on} \ \partial\Omega\times(0,T),
\end{split}
\end{equation}
where $0<T\leq\infty$, $\kappa \in C^{1}(\overline{\Omega})$ satisfies $\kappa>0$ on $\overline{\Omega}$, $\kappa_{s} \in C^{1}(\partial\Omega)$ satisfies $\kappa_{s}>0$ on $\partial\Omega$, $u_{0}$, $f$ and $g$ are arbitrarily given functions in $\Omega$, in $\Omega\times(0,T)$ and on $\partial\Omega\times(0,T)$ respectively, $\nu \in C^{2}(\partial\Omega)$ is the outward unit normal vector on $\partial\Omega$.
The unique solvability of (1.1) in the anisotropic Sobolev space $W^{2,1}_{p,q}(\Omega\times(0,T))$ $(1<p<\infty, \ 1<q<\infty)$ has been studied in recent years.
Weidemaier \cite{Weidemaier} discussed the unique solvability of (1.1) in $W^{2,1}_{p,q}(\Omega\times(0,T))$ for any $0<T<\infty$, $3/2<p\leq q<\infty$, which remains to consider the case where $T=\infty$, $1<p<\infty$, $1<q<\infty$.
The main purpose of this paper is to obtain the unique solvability of (1.1) in $W^{2,1}_{p,q}(\Omega\times(0,T))$ by the argument based on \cite{Shibata 2}.
First, (1.1) has uniquely a solution in $W^{2,1}_{p,q}(\Omega\times(0,T))$ for any $0<T\leq\infty$, $1<p<\infty$, $1<q<\infty$.
Second, exponentially weighted $L_{p}$-$L_{q}$ estimates for time-global solutions to (1.1) are established.

This paper is organized as follows: In section 2, we define basic notation used in this paper, and state our main results and some lemmas for them.
$L_{p}$-$L_{q}$ estimates for time-global solutions to (1.1) are established in section 3.
In section 4, we prove not only the unique solvability of (1.1) in $W^{2,1}_{p,q}(\Omega\times(0,T))$ for any $0<T\leq\infty$, $1<p<\infty$, $1<q<\infty$ but also exponentially weighted $L_{p}$-$L_{q}$ estimates for time-global solutions to (1.1).
Finally, as a simple application of $L_{p}$-$L_{q}$ estimates for time-global solutions to (1.1), we consider the following initial-boundary value problem for the semilinear heat equation:
\begin{equation}
\begin{split}
&\partial_{t}u-\mathrm{div}(\kappa\nabla u)-|u|^{r-1}u=f & \mathrm{in} \ \Omega\times(0,T), \\
&u|_{t=0}=u_{0} & \mathrm{in} \ \Omega, \\
&\kappa\partial_{\nu}u+\kappa_{s}u|_{\partial\Omega}=g & \mathrm{on} \ \partial\Omega\times(0,T),
\end{split}
\end{equation}
where $r>1$.
It is proved in section 5 that (1.2) has uniquely a solution in $W^{2,1}_{p,q}(\Omega\times(0,T))$ for any $0<T\leq\infty$, $1<p<\infty$, $2<q<\infty$ and that $L_{p}$-$L_{q}$ estimates for time-local and time-global solutions to (1.2) are established.

\section{Preliminaries and main results}
\subsection{Function spaces}
Function spaces and basic notation which we use throughout this paper are introduced as follows: Let $G$ be an open set in $\mathbb{R}^{n}$.
$(L_{p}(G),\|\cdot\|_{L_{p}(G)})$ and $(W^{k}_{p}(G),\|\cdot\|_{W^{k}_{p}(G)})$ $(1\leq p\leq\infty, \ k \in \mathbb{Z}, \ k\geq 0)$ are the Lebesgue space and the Sobolev space respectively, $W^{0}_{p}(G)=L_{p}(G)$.
$B^{r}_{p,q}(G)$ $(1<p<\infty, \ 1\leq q<\infty, \ 0<r<\infty)$ is the Besov space \cite[7.30--7.49]{Adams} defined as $B^{r}_{p,q}(G)=(L_{p}(G),W^{\langle r \rangle}_{p}(G))_{r/\langle r \rangle,q}$, where $(X_{0},X_{1})_{\theta,q}$ $(0<\theta<1)$ is an interpolation space between two Banach spaces $X_{0}$ and $X_{1}$ by the K-method or the J-method, $\langle r \rangle=\min\{k \in \mathbb{Z} \ ; \ k>r\}$.
$(B^{r}_{p,q,0}(G),\|\cdot\|_{B^{r}_{p,q,0}(G)})$ is the Banach space of all functions which are in $B^{r}_{p,q}(\mathbb{R}^{n})$ and have support on $\overline{G}$, $\|u\|_{B^{r}_{p,q,0}(G)}=\|u\|_{B^{r}_{p,q}(G)}+\|t^{-(r+1/q)}\|u\|_{L_{p}(G^{t})}\|_{L_{q}(\mathbb{R}_{+})}$, where $G^{t}:=\{x \in G \ ; \ \mathrm{dist}(x,\partial G)<t \}$ for any $t>0$.

Let $I$ be an open interval in $\mathbb{R}$, $(X,\|\cdot\|_{X})$ be a Banach space.
$(L_{q}(I;X),\|\cdot\|_{L_{q}(I;X)})$ and $(W^{l}_{q}(I;X),\|\cdot\|_{W^{l}_{q}(I;X)})$ $(1\leq q\leq\infty, \ l \in \mathbb{Z}, \ l\geq 0)$ are the Lebesgue space and the Sobolev space of $X$-valued functions respectively, $W^{0}_{q}(I;X)=L_{q}(I;X)$.
$(W^{k,l}_{p,q}(G\times I),\|\cdot\|_{W^{k,l}_{p,q}(G\times I)})$ $(1\leq p\leq\infty, \ k \in \mathbb{Z}, \ k\geq 0)$ is the anisotropic Sobolev space defined as
\begin{equation*}
W^{k,l}_{p,q}(G\times I)=L_{q}(I;W^{k}_{p}(G))\cap W^{l}_{q}(I;L_{p}(G)),
\end{equation*}
\begin{equation*}
\|u\|_{W^{k,l}_{p,q}(G\times I)}=\|u\|_{L_{q}(I;W^{k}_{p}(G))}+\|u\|_{W^{l}_{q}(I;L_{p}(G))}.
\end{equation*}
In the case where $I=(0,T)$, we introduce the following function spaces:
\begin{equation*}
W^{l}_{q,0}((0,T);X)=\{u \in W^{l}_{q}((-\infty,T);X) \ ; \ u(t)=0 \ \mathrm{for \ any} \ t<0\},
\end{equation*}
\begin{equation*}
W^{0}_{q,0}((0,T);X)=L_{q,0}((0,T);X),
\end{equation*}
\begin{equation*}
W^{k,l}_{p,q,0}(G\times(0,T))=\{u \in W^{k,l}_{p,q}(G\times(-\infty,T)) \ ; \ u(t)=0 \ \mathrm{for \ any} \ t<0\}.
\end{equation*}

The Fourier transform $\mathcal{F}_{x}[f(x)](\xi)$, the inverse Fourier transform $\mathcal{F}^{-1}_{\xi}[g(\xi)](x)$, the Fourier-Laplace transform $\mathcal{L}_{(x,t)}[f(x,t)](\xi,\lambda)$ and the inverse Fourier-Laplace transform $\mathcal{L}^{-1}_{(\xi,\lambda_{2})}[g(\xi,\lambda)](x,t)$ are defined as follows:
\begin{equation*}
\mathcal{F}_{x}[f(x)](\xi)=\frac{1}{(2\pi)^{n/2}}\int_{\mathbb{R}^{n}}e^{-\sqrt{-1}x\cdot\xi}f(x)dx
\end{equation*}
for any $\xi \in \mathbb{R}^{n}$,
\begin{equation*}
\mathcal{F}^{-1}_{\xi}[g(\xi)](x)=\frac{1}{(2\pi)^{n/2}}\int_{\mathbb{R}^{n}}e^{\sqrt{-1}x\cdot\xi}g(\xi)d\xi
\end{equation*}
for any $x \in \mathbb{R}^{n}$,
\begin{equation*}
\mathcal{L}_{(x,t)}[f(x,t)](\xi,\lambda)=\frac{1}{(2\pi)^{(n+1)/2}}\int_{\mathbb{R}}e^{-\lambda t}dt\int_{\mathbb{R}^{n}}e^{-\sqrt{-1}x\cdot\xi}f(x,t)dx
\end{equation*}
for any $\xi \in \mathbb{R}^{n}$, $\lambda \in \mathbb{C}$, $\lambda=\lambda_{1}+\sqrt{-1}\lambda_{2}$, $\lambda_{1}, \lambda_{2} \in \mathbb{R}$,
\begin{equation*}
\mathcal{L}^{-1}_{(\xi,\lambda_{2})}[g(\xi,\lambda)](x,t)=\frac{1}{(2\pi)^{(n+1)/2}}\int_{\mathbb{R}}e^{\lambda t}d\lambda_{2}\int_{\mathbb{R}^{n}}e^{\sqrt{-1}x\cdot\xi}g(\xi,\lambda)d\xi
\end{equation*}
for any $x \in \mathbb{R}^{n}$, $t \in \mathbb{R}$.
The Fourier transform $\mathcal{F}_{t}[f(t)](\lambda_{2})$ with respect to $t \in \mathbb{R}$ and the inverse Fourier transform $\mathcal{F}^{-1}_{\lambda_{2}}[g(\lambda_{2})](t)$ with respect to $\lambda_{2} \in \mathbb{R}$ are similarly defined.
It can be easily seen that
\begin{equation}
\mathcal{L}_{(x,t)}[f(x,t)](\xi,\lambda)=\mathcal{F}_{(x,t)}[e^{-\lambda_{1}t}f(x,t)](\xi,\lambda_{2})
\end{equation}
for any $\xi \in \mathbb{R}^{n}$, $\lambda \in \mathbb{C}$, $\lambda=\lambda_{1}+\sqrt{-1}\lambda_{2}$, $\lambda_{1}, \lambda_{2} \in \mathbb{R}$,
\begin{equation}
\mathcal{L}^{-1}_{(\xi,\lambda_{2})}[g(\xi,\lambda)](x,t)=e^{\lambda_{1}t}\mathcal{F}^{-1}_{(\xi,\lambda_{2})}[g(\xi,\lambda)](x,t)
\end{equation}
for any $x \in \mathbb{R}^{n}$, $t \in \mathbb{R}$, $\lambda_{1} \in \mathbb{R}$.
$(H^{s}_{q}(\mathbb{R};X),\|\cdot\|_{H^{s}_{q}(\mathbb{R};X)})$ $(1\leq q\leq \infty, \ s\geq 0)$ is the Bessel-potential space defined as
\begin{equation*}
H^{s}_{q}(\mathbb{R};X)=\{u \in L_{q}(\mathbb{R};X) \ ; \ \langle d_{t} \rangle^{s}u \in L_{q}(\mathbb{R};X)\},
\end{equation*}
\begin{equation*}
\langle d_{t} \rangle^{s}u(t)=\mathcal{F}^{-1}_{\lambda_{2}}[(1+\lambda^{2}_{2})^{s/2}\mathcal{F}_{t}[u(t)](\lambda_{2})](t),
\end{equation*}
\begin{equation*}
\|u\|_{H^{s}_{q}(\mathbb{R};X)}=\|u\|_{L_{q}(\mathbb{R};X)}+\|\langle d_{t} \rangle^{s}u\|_{L_{q}(\mathbb{R};X)},
\end{equation*}
\begin{equation*}
H^{s}_{q}((0,T);X)=\{u \in L_{q}((0,T);X) \ ; \ \exists v \in H^{s}_{q}(\mathbb{R};X), \ v(t)=u(t) \ \mathrm{for \ any} \ 0<t<T\},
\end{equation*}
\begin{equation*}
\|u\|_{H^{s}_{q}((0,T);X)}=\inf\{\|v\|_{H^{s}_{q}(\mathbb{R};X)} \ ; \ \forall v \in H^{s}_{q}(\mathbb{R};X), \ v(t)=u(t) \ \mathrm{for \ any} \ 0<t<T\},
\end{equation*}
\begin{equation*}
H^{k,s}_{p,q}(G\times\mathbb{R})=L_{q}(\mathbb{R};W^{k}_{p}(G))\cap H^{s}_{q}(\mathbb{R};L_{p}(G)),
\end{equation*}
\begin{equation*}
\|u\|_{H^{k,s}_{p,q}(G\times\mathbb{R})}=\|u\|_{L_{q}(\mathbb{R};W^{k}_{p}(G))}+\|u\|_{H^{s}_{q}(\mathbb{R};L_{p}(G))},
\end{equation*}
\begin{equation*}
H^{k,s}_{p,q,0}(G\times\mathbb{R}_{+})=\{u \in H^{k,s}_{p,q}(G\times\mathbb{R}) \ ; \ u(t)=0 \ \mathrm{for \ any} \ t<0\},
\end{equation*}
\begin{equation*}
H^{k,s}_{p,q,0}(G\times(0,T))=\{u \in L_{q}((0,T);L_{p}(G)) \ ; \ \exists v \in H^{k,s}_{p,q,0}(G\times\mathbb{R}_{+}), \ v(t)=u(t) \ \mathrm{for \ any} \ 0<t<T\},
\end{equation*}
\begin{equation*}
\|u\|_{H^{k,s}_{p,q,0}(G\times(0,T))}=\inf\{\|v\|_{H^{k,s}_{p,q}(G\times\mathbb{R})} \ ; \ \forall v \in H^{k,s}_{p,q,0}(G\times\mathbb{R}_{+}), \ v(t)=u(t) \ \mathrm{for \ any} \ 0<t<T\}.
\end{equation*}
$\mathcal{S}(\mathbb{R};X)$ is the set of all $X$-valued functions which are rapidly decreasing in $\mathbb{R}$.

Let $(X,\|\cdot\|_{X})$ and $(Y,\|\cdot\|_{Y})$ be Banach spaces.
$(\mathcal{B}(X;Y),\|\cdot\|_{\mathcal{B}(X;Y)})$ is the Banach space of all bounded linear operators from $X$ to $Y$, $\mathcal{B}(X)=\mathcal{B}(X;X)$,
\begin{equation*}
\|A\|_{\mathcal{B}(X;Y)}=\sup_{x \in X\setminus\{0\}}\frac{\|Ax\|_{Y}}{\|x\|_{X}}.
\end{equation*}

$(l^{s}_{q},\|\cdot\|_{l^{s}_{q}})$ $(1\leq q\leq \infty, \ s\geq 0)$ is the Banach space of all $\mathbb{C}$-valued sequences defined as
\begin{equation*}
l^{s}_{q}=\{(a_{j})_{j \in \mathbb{Z}} \ ; \ a_{j} \in \mathbb{C} \ (j \in \mathbb{Z}), \ \|(a_{j})_{j \in \mathbb{Z}}\|_{l^{s}_{q}}<\infty\},
\end{equation*}
\begin{equation*}
\|(a_{j})_{j \in \mathbb{Z}}\|_{l^{s}_{q}}=
\begin{cases}
\displaystyle \left(\sum_{j \in \mathbb{Z}}|2^{js}a_{j}|^{q}\right)^{1/q} & \mathrm{if} \ 1\leq q<\infty, \\
\displaystyle \sup_{j \in \mathbb{Z}}2^{js}|a_{j}| & \mathrm{if} \ q=\infty.
\end{cases}
\end{equation*}
It follows from \cite[Theorem 5.6.1]{Bergh} that $l^{s}_{q}=(l^{1}_{\infty},l^{0}_{\infty})_{s,q}$ for any $1<q<\infty$, $0<s<1$.

\subsection{Strongly elliptic operator in $L_{p}$}
The strongly elliptic operator $A_{p}$ $(1<p<\infty)$ in $L_{p}(\Omega)$ with the zero Robin boundary condition is defined as $A_{p}=-\mathrm{div}(\kappa\nabla\cdot)$, $\mathcal{D}(A_{p})=\{u \in W^{2}_{p}(\Omega) \ ; \ \kappa\partial_{\nu}u+\kappa_{s}u|_{\partial\Omega}=0\}$, where $\mathcal{D}(A_{p})$ is the domain of $A_{p}$.
It is the same as in \cite[Theorems 2.5.2 and 7.3.6]{Pazy} that $A_{p}$ is a sectorial operator in $L_{p}(\Omega)$.
Therefore, $-A_{p}$ generates an uniformly bounded analytic semigroup $\{e^{-tA_{p}}\}_{t\geq 0}$ on $L_{p}(\Omega)$, fractional powers $A^{\alpha}_{p}$ of $A_{p}$ can be defined for any $\alpha\geq 0$, $A^{0}_{p}=I_{p}$, where $I_{p}$ is the identity operator in $L_{p}(\Omega)$.
Let us introduce Banach spaces derived from $A^{\alpha}_{p}$.
$(X^{\alpha}_{p}(\Omega),\|\cdot\|_{X^{\alpha}_{p}(\Omega)})$ and $X_{p,q}(\Omega)$ $(1<q<\infty)$ are defined as $X^{\alpha}_{p}(\Omega)=\mathcal{D}(A^{\alpha}_{p})$ with the norm $\|u\|_{X^{\alpha}_{p}(\Omega)}=\|A^{\alpha}_{p}u\|_{L_{p}(\Omega)}$ and $X_{p,q}(\Omega)=(L_{p}(\Omega),X^{1}_{p}(\Omega))_{1-1/q,q}$ respectively.
$\Lambda_{1}$ is the first eigenvalue of $-\mathrm{div}(\kappa\nabla\cdot)$ with the zero Robin boundary condition.

We state some lemmas concerning sectorial operators in Banach spaces and the characterization of $X_{p,q}(\Omega)$.
See, for example, \cite[Chapter 1]{Henry}, \cite[Chapter 2]{Pazy} about the theory of analytic semigroups on Banach spaces and fractional powers of sectorial operators.
\begin{lemma}
Let $1<p<\infty$, $\alpha\geq 0$, $0<\lambda_{1}<\Lambda_{1}$.
Then
\begin{equation}
\|A^{\alpha}_{p}e^{-tA_{p}}\|_{\mathcal{B}(L_{p}(\Omega))}\leq Ct^{-\alpha}e^{-\lambda_{1}t},
\end{equation}
where $C$ is a positive constant depending only on $n$, $\Omega$, $p$, $\kappa$, $\kappa_{s}$, $\alpha$ and $\lambda_{1}$.
\end{lemma}
\begin{proof}
It is \cite[Theorem 1.4.3]{Henry}.
\end{proof}
\begin{lemma}
Let $1<p<\infty$, $0\leq\alpha\leq 1$.
Then
\begin{equation}
X^{\alpha}_{p}(\Omega)\hookrightarrow W^{k}_{r}(\Omega) \ \mathrm{if} \ \frac{1}{p}-\frac{2\alpha-k}{n}\leq\frac{1}{r}\leq\frac{1}{p},
\end{equation}
where $\hookrightarrow$ is the continuous inclusion.
\end{lemma}
\begin{proof}
It is \cite[Theorem 1.6.1]{Henry}.
\end{proof}
\begin{lemma}
Let $1<p<\infty$, $1<q<\infty$.
Then
\begin{equation}
X_{p,q}(\Omega)=
\begin{cases}
\{u \in B^{2(1-1/q)}_{p,q}(\Omega) \ ; \ \kappa\partial_{\nu}u+\kappa_{s}u|_{\partial\Omega}=0\} & \mathrm{if} \ 2(1-1/q)>1+1/p, \\
\{u \in B^{2(1-1/q)}_{p,q}(\Omega) \ ; \ \kappa\partial_{\tilde{\nu}}u+\tilde{\kappa}_{s}u \in B^{1/p}_{p,q,0}(\Omega)\} & \mathrm{if} \ 2(1-1/q)=1+1/p, \\
B^{2(1-1/q)}_{p,q}(\Omega) & \mathrm{if} \ 2(1-1/q)<1+1/p,
\end{cases}
\end{equation}
where $\tilde{\kappa}_{s} \in C^{1}(\overline{\Omega})$ and $\tilde{\nu} \in C^{2}(\overline{\Omega})$ are extensions of $\kappa_{s}$ and $\nu$ to $\overline{\Omega}$ respectively.
\end{lemma}
\begin{proof}
It is \cite[Proposition 1.25 and Theorem 3.5]{Guidetti}.
\end{proof}

\subsection{Main results}
First, we shall prove the unique solvability of (1.1) in $W^{2,1}_{p,q}(\Omega\times(0,T))$ for any $0<T<\infty$, $1<p<\infty$, $1<q<\infty$ and $L_{p}$-$L_{q}$ estimates for solutions to (1.1) uniformly in $T$.
\begin{theorem}
Let $\Omega$ be a bounded domain in $\mathbb{R}^{n}$ with its $C^{2}$-boundary $\partial\Omega$, $\kappa \in C^{1}(\overline{\Omega})$ satisfy $\kappa>0$ on $\overline{\Omega}$, $\kappa_{s} \in C^{1}(\partial\Omega)$ satisfy $\kappa_{s}>0$ on $\partial\Omega$, $0<T<\infty$, $1<p<\infty$, $1<q<\infty$, $u_{0} \in X_{p,q}(\Omega)$, $f \in L_{q}((0,T);L_{p}(\Omega))$, $g \in H^{1,1/2}_{p,q,0}(\Omega\times(0,T))$.
Then $(1.1)$ has uniquely a solution $u \in W^{2,1}_{p,q}(\Omega\times(0,T))$ satisfying
\begin{equation}
\|u\|_{W^{2,1}_{p,q}(\Omega\times(0,T))}\leq C_{p,q}(\|u_{0}\|_{X_{p,q}(\Omega)}+\|f\|_{L_{q}((0,T);L_{p}(\Omega))}+\|g\|_{H^{1,1/2}_{p,q,0}(\Omega\times(0,T))}),
\end{equation}
where $C_{p,q}$ is a positive constant depending only on $n$, $\Omega$, $p$, $q$, $\kappa$ and $\kappa_{s}$.
\end{theorem}
Second, in addition to Theorem 2.1 with $T=\infty$, exponentially weighted $L_{p}$-$L_{q}$ estimates for time-global solutions to (1.1) can be established as follows:
\begin{theorem}
Let $\Omega$ be a bounded domain in $\mathbb{R}^{n}$ with its $C^{2}$-boundary $\partial\Omega$, $\kappa \in C^{1}(\overline{\Omega})$ satisfy $\kappa>0$ on $\overline{\Omega}$, $\kappa_{s} \in C^{1}(\partial\Omega)$ satisfy $\kappa_{s}>0$ on $\partial\Omega$, $1<p<\infty$, $1<q<\infty$, $u_{0} \in X_{p,q}(\Omega)$.
Then there exists a positive constant $\lambda^{0}_{1}$ depending only on $n$, $\Omega$, $p$, $q$, $\kappa$ and $\kappa_{s}$ such that if $e^{\lambda_{1}t}f \in L_{q}(\mathbb{R}_{+};L_{p}(\Omega))$, $e^{\lambda_{1}t}g \in H^{1,1/2}_{p,q,0}(\Omega\times\mathbb{R}_{+})$ for some $0\leq\lambda_{1}\leq \lambda^{0}_{1}$, then $(1.1)$ has uniquely a solution $u \in W^{2,1}_{p,q}(\Omega\times\mathbb{R}_{+})$ satisfying
\begin{equation}
\|e^{\lambda_{1}t}u\|_{W^{2,1}_{p,q}(\Omega\times\mathbb{R}_{+})}\leq C_{p,q,\lambda_{1}}(\|u_{0}\|_{X_{p,q}(\Omega)}+\|e^{\lambda_{1}t}f\|_{L_{q}(\mathbb{R}_{+};L_{p}(\Omega))}+\|e^{\lambda_{1}t}g\|_{H^{1,1/2}_{p,q,0}(\Omega\times\mathbb{R}_{+})}),
\end{equation}
where $C_{p,q,\lambda_{1}}$ is a positive constant depending only on $n$, $\Omega$, $p$, $q$, $\kappa$, $\kappa_{s}$ and $\lambda_{1}$.
\end{theorem}

\subsection{Lemmas}
We will state some lemmas which play an important role throughout this paper.
First, the generalized Bochner theorem is stated as follows:
\begin{lemma}
Let $(X,\|\cdot\|_{X})$ be a Banach space, $a=m+\mu-n$, where $m, n \in \mathbb{Z}$, $m\geq 0$, $n\geq 2$, $0<\mu\leq 1$.
Assume that $g \in C^{\infty}(\mathbb{R}^{n}\setminus\{0\};X)$ satisfies
\begin{equation}
\partial^{\gamma}g \in L_{1}(\mathbb{R}^{n};X)
\end{equation}
for any $\gamma \in \mathbb{Z}^{n}$, $\gamma\geq 0$, $|\gamma|\leq m$,
\begin{equation}
\|\partial^{\gamma}_{\xi}g(\xi)\|_{X}\leq C_{\gamma}|\xi|^{a-|\gamma|}
\end{equation}
for any $\xi \in \mathbb{R}^{n}\setminus\{0\}$, $\gamma \in \mathbb{Z}^{n}$, $\gamma\geq 0$, where $C_{\gamma}$ is a positive constant.
Then
\begin{equation}
\|\mathcal{F}^{-1}_{\xi}[g(\xi)](x)\|_{X}\leq C\left(\max_{\gamma \in \mathbb{Z}^{n}, \gamma\geq 0, |\gamma|\leq m+2}C_{\gamma}\right)|x|^{-(m+\mu)}
\end{equation}
for any $x \in \mathbb{R}^{n}\setminus\{0\}$, where $C$ is a positive constant depending only on $m$, $n$ and $\mu$.
\end{lemma}
\begin{proof}
It is \cite[Theorem 2.3]{Shibata 1}.
\end{proof}
Second, in order to give the operator-valued Fourier multiplier theorem, we introduce an UMD space and the $\mathcal{R}$-boundedness.
A Banach space $(X,\|\cdot\|_{X})$ is called an UMD space if $X$ satisfies the following property (HT) which is equivalent to the UMD property \cite[Theorem]{Bourgain}, \cite[Theorem 2]{Burkholder}:
\begin{itemize}
\item[(HT)]The Hilbert transform
\begin{equation*}
Hf(t)=\frac{1}{\pi}\mathrm{p.v.}\int_{\mathbb{R}}\frac{f(s)}{t-s}ds, \ \mathcal{D}(H)=\mathcal{S}(\mathbb{R};X)
\end{equation*}
is extended to a bounded linear operator $H \in \mathcal{B}(L_{q}(\mathbb{R};X))$ for some $1<q<\infty$.
\end{itemize}
It is well known in \cite[Condition (iii)]{Burkholder} that typical examples of UMD spaces are $l_{p}$ and $L_{p}(G)$ for any $1<p<\infty$.
Let $\mathcal{T}$ be a subset of $\mathcal{B}(X;Y)$.
$\mathcal{T}$ is called $\mathcal{R}$-bounded on $\mathcal{B}(X;Y)$ if $\mathcal{T}$ satisfies the following property $(\mathcal{R})$:
\begin{itemize}
\item[$(\mathcal{R})$]There exists a positive constant $C$ such that
\begin{equation*}
\int^{1}_{0}\left\|\sum^{N}_{j=1}r_{j}(t)T_{j}x_{j}\right\|_{Y}dt\leq C\int^{1}_{0}\left\|\sum^{N}_{j=1}r_{j}(t)x_{j}\right\|_{X}dt
\end{equation*}
for any $T_{j} \in \mathcal{T}$, $x_{j} \in X$ $(j=1, \cdots, N)$, $N \in \mathbb{Z}$, $N\geq 1$, where $r_{j}$ is the Rademacher function defined as $r_{j}(t)=\mathrm{sign}(\sin(2^{j}\pi t))$.
\end{itemize}
In the case where $\mathcal{T}$ is $\mathcal{R}$-bounded on $\mathcal{B}(X;Y)$, the smallest constant $C$ for which $(\mathcal{R})$ holds is denoted by $\mathcal{R}(\mathcal{T})$.
The operator-valued Fourier multiplier theorem on UMD spaces is stated as follows:
\begin{lemma}
Let $(X,\|\cdot\|_{X})$ and $(Y,\|\cdot\|_{Y})$ be UMD spaces, $M \in C^{1}(\mathbb{R}\setminus\{0\};\mathcal{B}(X;Y))$.
Assume that $\{M(t)\}_{t \in \mathbb{R}\setminus\{0\}}$ and $\{td_{t}M(t)\}_{t \in \mathbb{R}\setminus\{0\}}$ are $\mathcal{R}$-bounded on $\mathcal{B}(X;Y)$,
\begin{equation*}
\mathcal{R}(\{M(t)\}_{t \in \mathbb{R}\setminus\{0\}})=C_{0},
\end{equation*}
\begin{equation*}
\mathcal{R}(\{td_{t}M(t)\}_{t \in \mathbb{R}\setminus\{0\}})=C_{1}.
\end{equation*}
Then $K=\mathcal{F}^{-1}M\mathcal{F}$, $\mathcal{D}(K)=C^{\infty}_{0}(\mathbb{R}\setminus\{0\};X)$ is extended to a bounded linear operator $K \in \mathcal{B}(L_{q}(\mathbb{R};X);L_{q}(\mathbb{R};Y))$ for any $1<q<\infty$,
\begin{equation}
\|K\|_{\mathcal{B}(L_{q}(\mathbb{R};X);L_{q}(\mathbb{R};Y))}\leq C(C_{0}+C_{1}),
\end{equation}
where $C$ is a positive constant depending only on $X$, $Y$ and $q$.
\end{lemma}
\begin{proof}
It is \cite[Theorem 3.4]{Weis}.
\end{proof}
It is very useful on families of convolution operators that we give the following sufficient condition for the $\mathcal{R}$-boundedness:
\begin{lemma}
Set $\mathcal{M}=\{M(t)\}_{t \in \mathbb{R}\setminus\{0\}}$,
\begin{equation*}
(M(t)f)(x)=\int_{\mathbb{R}^{n}}P(x-y,t)f(y)dy, \ P(\cdot,t) \in L_{1,loc}(\mathbb{R}^{n}), \ \mathcal{D}(M(t))=L_{2}(\mathbb{R}^{n}).
\end{equation*}
Assume that there exists a positive constant $C_{0}$ such that
\begin{equation}
\|M(t)\|_{\mathcal{B}(L_{2}(\mathbb{R}^{n}))}\leq C_{0}
\end{equation}
for any $t \in \mathbb{R}\setminus\{0\}$,
\begin{equation}
\sum_{\beta \in \mathbb{Z}^{n}, \beta\geq 0, |\beta|=1}|\partial^{\beta}_{x}P(x,t)|\leq C_{0}|x|^{-(n+1)}
\end{equation}
for any $x \in \mathbb{R}^{n}\setminus\{0\}$, $t \in \mathbb{R}\setminus\{0\}$.
Then $\mathcal{M}$ is $\mathcal{R}$-bounded on $\mathcal{B}(L_{p}(\mathbb{R}^{n}))$ for any $1<p<\infty$.
Moreover, $\mathcal{R}(\mathcal{M})\leq CC_{0}$, where $C$ is a positive constant depending only on $n$ and $p$.
\end{lemma}
\begin{proof}
It is \cite[Proposition 2.4]{Shibata 2}.
\end{proof}
Third, some interpolation inequalities in Bessel-potential spaces are stated as follows:
\begin{lemma}
Let $G$ be an open set in $\mathbb{R}^{n}$, $1<p<\infty$, $1<q<\infty$, $0<s<1$, $0<r\leq 1$.
Then
\begin{equation}
\|u\|_{H^{s}_{q}(\mathbb{R};L_{p}(G))}\leq C(r^{1-s}\|\partial_{t}u\|_{L_{q}(\mathbb{R};L_{p}(G))}+r^{-s}\|u\|_{L_{q}(\mathbb{R};L_{p}(G))}),
\end{equation}
\begin{equation}
\|u\|_{H^{s}_{q}(\mathbb{R};L_{p}(G))}\leq C\|u\|^{s}_{H^{1}_{q}(\mathbb{R};L_{p}(G))}\|u\|^{1-s}_{L_{q}(\mathbb{R};L_{p}(G))}
\end{equation}
for any $u \in H^{1}_{q}(\mathbb{R};L_{p}(G))$, where $C$ is a positive constant depending only on $G$, $q$ and $s$.
\end{lemma}
\begin{proof}
It is \cite[Proposition 2.6]{Shibata 2}.
\end{proof}
\begin{lemma}
Let $1<p<\infty$, $1<q<\infty$.
Then
\begin{equation}
\|u\|_{H^{1/2}_{q}(\mathbb{R};W^{1}_{p}(\Omega))}\leq C\|u\|_{W^{2,1}_{p,q}(\Omega\times\mathbb{R})}
\end{equation}
for any $u \in W^{2,1}_{p,q}(\Omega\times\mathbb{R})$, where $C$ is a positive constant depending only on $n$, $\Omega$, $p$ and $q$.
\end{lemma}
\begin{proof}
It is \cite[Proposition 2.8]{Shibata 2}.
\end{proof}
Finally, we state some lemmas which is essential for $L_{p}$-$L_{q}$ estimates for time-global solutions to (1.1) in $\mathbb{R}^{n}_{+}\times\mathbb{R}$ with a positive constant $\kappa$.
Let us introduce subsets with respect to $\xi' \in \mathbb{R}^{n-1}$, $x_{n} \in \mathbb{R}$ and $\lambda \in \mathbb{C}$ as follows:
\begin{equation*}
G_{(\xi',\lambda)}=\{(\xi',\lambda) \in \mathbb{R}^{n-1}\times\mathbb{C} \ ; \ \xi' \in \mathbb{R}^{n-1}\setminus\{0\}, \ \lambda=\lambda_{1}+\sqrt{-1}\lambda_{2}, \ \lambda_{1}\geq 0, \ \lambda_{2} \in \mathbb{R}\setminus\{0\}\},
\end{equation*}
\begin{equation*}
G_{(\xi',x_{n},\lambda)}=\{(\xi',x_{n},\lambda) \in \mathbb{R}^{n-1}\times\mathbb{R}\times\mathbb{C} \ ; \ (\xi',\lambda) \in G_{(\xi',\lambda)}, \ x_{n}\geq 0\}.
\end{equation*}
$A(\xi',\lambda)$ is one of characteristic roots for $\kappa|\xi'|^{2}+\lambda+1-\kappa A^{2}=0$ defined as 
\begin{equation*}
A(\xi',\lambda)=(|\xi'|^{2}+\kappa^{-1}(\lambda+1))^{1/2}.
\end{equation*}
It can be easily seen that
\begin{equation*}
\mathrm{Re}A(\xi',\lambda)\geq c(|\xi'|+|\lambda|^{1/2}+1)
\end{equation*}
for any $(\xi',\lambda) \in G_{(\xi',\lambda)}$, where $c$ is a positive constant depending only on $\kappa$.
\begin{lemma}
\begin{equation}
|\partial^{\alpha'}_{\xi'}A(\xi',\lambda)^{s}|\leq C(|\xi'|+|\lambda|^{1/2}+1)^{s-|\alpha'|}, \ C=C(s,\alpha'),
\end{equation}
\begin{equation}
|\partial^{\alpha'}_{\xi'}|\xi'|^{s}|\leq C|\xi'|^{s-|\alpha'|}, \ C=C(s,\alpha'),
\end{equation}
\begin{equation}
|\partial^{\alpha'}_{\xi'}e^{-A(\xi',\lambda)x_{n}}|\leq C(|\xi'|+|\lambda|^{1/2}+1)^{-|\alpha'|}e^{-d(|\xi'|+|\lambda|^{1/2}+1)x_{n}}, \ C=C(\alpha')
\end{equation}
for any $(\xi',x_{n},\lambda) \in G_{(\xi',x_{n},\lambda)}$, $s \in \mathbb{R}$, $\alpha' \in \mathbb{Z}^{n-1}$, $\alpha'\geq 0$, where $C$ is a positive constant depending only on $\kappa$, $s$ and $\alpha'$, $d$ is a positive constant depending only on $\kappa$.
\end{lemma}
\begin{proof}
It is \cite[Lemma 5.4]{Shibata 2}.
\end{proof}
\begin{lemma}
Let $1<p<\infty$, $1<q<\infty$, $M_{1}$ be a function defined on $G_{(\xi',\lambda)}$.
Assume that there exist constants $C>0$ and $\beta\leq 1$ such that
\begin{equation}
|\partial^{\alpha'}_{\xi'}(\lambda^{l}_{2}\partial^{l}_{\lambda_{2}}M_{1}(\xi',\lambda))|\leq C(|\xi'|+|\lambda|^{1/2}+1)^{-|\alpha'|}(|\lambda|+1)^{\beta}, \ C=C(\alpha')
\end{equation}
for any $(\xi',\lambda) \in G_{(\xi',\lambda)}$, $l=0, 1$, $\alpha' \in \mathbb{Z}^{n-1}$, $\alpha'\geq 0$.
Let $\hat{f}$ be a given function defined on $G_{(\xi',x_{n},\lambda)}$,
\begin{equation*}
w_{1}=\int^{\infty}_{0}\mathcal{L}^{-1}_{(\xi',\lambda_{2})}[A(\xi',\lambda)^{-1}e^{-A(\xi',\lambda)(x_{n}+y_{n})}M_{1}(\xi',\lambda)\hat{f}(\xi',y_{n},\lambda)](x',t)dy_{n},
\end{equation*}
\begin{equation*}
f(x,t)=\mathcal{L}^{-1}_{(\xi',\lambda_{2})}[\hat{f}(\xi',x_{n},\lambda)](x',t).
\end{equation*}
Then
\begin{equation}
\|e^{-\lambda_{1}t}w_{1}\|_{L_{q}(\mathbb{R};L_{p}(\mathbb{R}^{n}_{+}))}\leq C\|e^{-\lambda_{1}t}f\|_{L_{q}(\mathbb{R};L_{p}(\mathbb{R}^{n}_{+}))}
\end{equation}
for any $\lambda_{1}\geq 0$, where $C$ is a positive constant depending only on $n$, $p$, $q$ and $\kappa$.
\end{lemma}
\begin{proof}
It is \cite[Lemma 5.5]{Shibata 2}.
\end{proof}
\begin{lemma}
Let $1<p<\infty$, $1<q<\infty$, $M_{2}$ be a function defined on $G_{(\xi',\lambda)}$.
Assume that there exists a positive constant $C$ such that
\begin{equation}
|\partial^{\alpha'}_{\xi'}(\lambda^{l}_{2}\partial^{l}_{\lambda_{2}}M_{2}(\xi',\lambda))|\leq C(|\xi'|+|\lambda|^{1/2}+1)^{-|\alpha'|}, \ C=C(\alpha')
\end{equation}
for any $(\xi',\lambda) \in G_{(\xi',\lambda)}$, $l=0, 1$, $\alpha' \in \mathbb{Z}^{n-1}$, $\alpha'\geq 0$.
Let $\hat{g}$ be a given function defined on $G_{(\xi',x_{n},\lambda)}$,
\begin{equation*}
w_{2}=\int^{\infty}_{0}\mathcal{L}^{-1}_{(\xi',\lambda_{2})}[A(\xi',\lambda)e^{-A(\xi',\lambda)(x_{n}+y_{n})}M_{2}(\xi',\lambda)\hat{g}(\xi',y_{n},\lambda)](x',t)dy_{n},
\end{equation*}
\begin{equation*}
g(x,t)=\mathcal{L}^{-1}_{(\xi',\lambda_{2})}[\hat{g}(\xi',x_{n},\lambda)](x',t).
\end{equation*}
Then
\begin{equation}
\|e^{-\lambda_{1}t}w_{2}\|_{L_{q}(\mathbb{R};L_{p}(\mathbb{R}^{n}_{+}))}\leq C\|e^{-\lambda_{1}t}g\|_{L_{q}(\mathbb{R};L_{p}(\mathbb{R}^{n}_{+}))}
\end{equation}
for any $\lambda_{1}\geq 0$, where $C$ is a positive constant depending only on $n$, $p$, $q$ and $\kappa$.
\end{lemma}
\begin{proof}
It is \cite[Lemmas 5.5 and 5.6]{Shibata 2}.
\end{proof}

\section{Global $L_{p}$-$L_{q}$ estimates for solutions to (1.1)}
\subsection{$L_{p}$ estimates for steady solutions to \rm{(1.1)}}
We discuss the following boundary value problem in $\Omega$:
\begin{equation}
\begin{split}
&\lambda u-\mathrm{div}(\kappa\nabla u)=f & \mathrm{in} \ \Omega, \\
&\kappa\partial_{\nu}u+\kappa_{s}u|_{\partial\Omega}=g & \mathrm{on} \ \partial\Omega
\end{split}
\end{equation}
for any $\lambda \in S_{\phi}\cup\{0\}$, where
\begin{equation*}
S_{\phi}=\{\lambda \in \mathbb{C}\setminus\{0\} \ ; \ |\mathrm{arg}\lambda|\leq \pi-\phi\}, \ 0<\phi<\frac{\pi}{2}.
\end{equation*}
It is essential for our main results that $L_{p}$ estimates for steady solutions to (1.1) are established as follows:
\begin{theorem}
Let $\kappa \in C^{1}(\overline{\Omega})$ satisfy $\kappa>0$ on $\overline{\Omega}$, $\kappa_{s} \in C^{1}(\partial\Omega)$ satisfy $\kappa_{s}>0$ on $\partial\Omega$, $1<p<\infty$, $0<\phi<\pi/2$, $f \in L_{p}(\Omega)$, $g \in W^{1}_{p}(\Omega)$.
Then $(3.1)$ has uniquely a solution $u \in W^{2}_{p}(\Omega)$ satisfying
\begin{equation}
|\lambda|\|u\|_{L_{p}(\Omega)}+\|u\|_{W^{2}_{p}(\Omega)}\leq C_{p}(\|f\|_{L_{p}(\Omega)}+\|g\|_{W^{1}_{p}(\Omega)})
\end{equation}
for any $\lambda \in S_{\phi}\cup\{0\}$, where $C_{p}$ is a positive constant depending only on $n$, $\Omega$, $p$, $\phi$, $\kappa$ and $\kappa_{s}$.
\end{theorem}
\begin{proof}
Set
\begin{equation*}
(u,v)_{\Omega}=\int_{\Omega}u(x)v(x)dx,
\end{equation*}
and let $p^{*}$ be the dual exponent to $p$ defined as $1/p+1/p^{*}=1$.
Then it follows from $1<p<\infty$ that $u|u|^{p-2} \in L_{p^{*}}(\Omega)$ for any $u \in L_{p}(\Omega)$ and that $(u,u|u|^{p-2})_{\Omega}=\|u\|^{p}_{L_{p}(\Omega)}$.
Let $2\leq p<\infty$.
Then integration by parts and $\lambda \in S_{\phi}\cup\{0\}$ yield that
\begin{equation*}
|(\lambda u+A_{p}u,u|u|^{p-2})_{\Omega}|\geq c\left(\int_{\Omega}|\nabla u(x)|^{2}|u(x)|^{p-2}dx+\|u\|^{p}_{L_{p}(\partial\Omega)}\right),
\end{equation*}
\begin{equation*}
c\left(\int_{\Omega}|\nabla u(x)|^{2}|u(x)|^{p-2}dx+\|u\|^{p}_{L_{p}(\partial\Omega)}\right)\leq \|(\lambda I_{p}+A_{p})u\|_{L_{p}(\Omega)}\|u\|^{p-1}_{L_{p}(\Omega)}
\end{equation*}
for any $u \in \mathcal{D}(A_{p})$, where $c$ is a positive constant depending only on $p$, $\phi$, $\kappa$ and $\kappa_{s}$ and that $(\lambda I_{p}+A_{p})u=0$ implies $\nabla u=0$, $u|_{\partial\Omega}=0$, that is, $u=0$.
Therefore, $\lambda I_{p}+A_{p}$ is injective.
In the case where $1<p<2$, it can be easily seen from the duality argument that $\lambda I_{p}+A_{p}$ is also injective.
By applying \cite[Theorem 15.2]{Agmon} to the following boundary value problem in $\Omega$:
\begin{equation*}
\begin{split}
&-\mathrm{div}(\kappa\nabla u)=f-\lambda u & \mathrm{in} \ \Omega, \\
&\kappa\partial_{\nu}u+\kappa_{s}u|_{\partial\Omega}=g & \mathrm{on} \ \partial\Omega,
\end{split}
\end{equation*}
it is obtained from a basic property of $A_{p}$ and the compactness-uniqueness argument that (3.2) is established.
Moreover, the adjoint problem \cite[Corollary to Theorem 5]{Browder} to the uniqueness of (3.1) admits that (3.1) has uniquely a solution $u \in W^{2}_{p}(\Omega)$ for any $1<p<\infty$. 
\end{proof}

\subsection{$L_{p}$-$L_{q}$ estimates in $\Omega\times\mathbb{R}_{+}$ with $f=0$, $g=0$}
We consider (1.1) with $f=0$, $g=0$, that is, the following initial-boundary value problem in $\Omega\times\mathbb{R}_{+}$:
\begin{equation}
\begin{split}
&\partial_{t}u-\mathrm{div}(\kappa\nabla u)=0 & \mathrm{in} \ \Omega\times\mathbb{R}_{+}, \\
&u|_{t=0}=u_{0} & \mathrm{in} \ \Omega, \\
&\kappa\partial_{\nu}u+\kappa_{s}u|_{\partial\Omega}=0 & \mathrm{on} \ \partial\Omega\times\mathbb{R}_{+}.
\end{split}
\end{equation}
The theory of analytic semigroups on $L_{p}$ admits that $L_{p}$-$L_{q}$ estimates for time-global solutions to (3.3) are established as follows:
\begin{theorem}
Let $\kappa \in C^{1}(\overline{\Omega})$ satisfy $\kappa>0$ on $\overline{\Omega}$, $\kappa_{s} \in C^{1}(\partial\Omega)$ satisfy $\kappa_{s}>0$ on $\partial\Omega$, $1<p<\infty$, $1<q<\infty$, $0<\lambda_{1}<\Lambda_{1}$, $u_{0} \in X_{p,q}(\Omega)$.
Then $(3.3)$ has uniquely a solution $u \in W^{2,1}_{p,q}(\Omega\times\mathbb{R}_{+})$ satisfying
\begin{equation*}
u(t)=e^{-tA_{p}}u_{0},
\end{equation*}
\begin{equation}
\|e^{(\lambda_{1}/2)t}u\|_{W^{2,1}_{p,q}(\Omega\times\mathbb{R}_{+})}\leq C_{p,q,\lambda_{1}}\|u_{0}\|_{X_{p,q}(\Omega)},
\end{equation}
where $C_{p,q,\lambda_{1}}$ is a positive constant depending only on $n$, $\Omega$, $p$, $q$, $\kappa$, $\kappa_{s}$ and $\lambda_{1}$.
\end{theorem}
\begin{proof}
It is sufficient for Theorem 3.2 to be proved that $u(t)=e^{-tA_{p}}u_{0}$ satisfies (3.4).
We can easily see from (2.3) with $\alpha=0$ that
\begin{equation*}
e^{(\lambda_{1}/2)t}\|u(t)\|_{L_{p}(\Omega)}\leq Ce^{-(\lambda_{1}/2)t}\|u_{0}\|_{L_{p}(\Omega)},
\end{equation*}
\begin{equation}
\|e^{(\lambda_{1}/2)t}u\|_{L_{q}(\mathbb{R}_{+};L_{p}(\Omega))}\leq C\|u_{0}\|_{L_{p}(\Omega)},
\end{equation}
where $C$ is a positive constant depending only on $n$, $\Omega$, $p$, $q$, $\kappa$, $\kappa_{s}$ and $\lambda_{1}$.
In order to obtain $L_{p}$-$L_{q}$ estimates for $e^{(\lambda_{1}/2)t}\partial_{t}u$, we calculate as follows:
\begin{equation*}
\begin{split}
\int^{\infty}_{0}e^{(q\lambda_{1}t)/2}\|\partial_{t}u(t)\|^{q}_{L_{p}(\Omega)}dt&=\sum_{j \in \mathbb{Z}}\int^{2^{j+1}}_{2^{j}}e^{(q\lambda_{1}t)/2}\|\partial_{t}u(t)\|^{q}_{L_{p}(\Omega)}dt \\
&\leq\sum_{j \in \mathbb{Z}}e^{q\lambda_{1}2^{j}}(2^{j+1}-2^{j})a_{j}(u_{0})^{q} \\
&=\sum_{j \in \mathbb{Z}}\left(2^{j/q}e^{\lambda_{1}2^{j}}a_{j}(u_{0})\right)^{q},
\end{split}
\end{equation*}
\begin{equation}
\|e^{(\lambda_{1}/2)t}\partial_{t}u\|_{L_{q}(\mathbb{R}_{+};L_{p}(\Omega))}\leq\|(e^{\lambda_{1}2^{j}}a_{j}(u_{0}))_{j \in \mathbb{Z}}\|_{l^{1/q}_{q}},
\end{equation}
where
\begin{equation*}
a_{j}(u_{0})=\max_{2^{j}\leq t\leq 2^{j+1}}\|\partial_{t}u(t)\|_{L_{p}(\Omega)}.
\end{equation*}
Since
\begin{equation*}
\|\partial_{t}u(t)\|_{L_{p}(\Omega)}\leq Ce^{-\lambda_{1}t}\|u_{0}\|_{X^{1}_{p}(\Omega)}
\end{equation*}
for any $t>0$, where $C$ is a positive constant depending only on $n$, $\Omega$, $p$, $\kappa$, $\kappa_{s}$ and $\lambda_{1}$, which follows from (2.3) with $\alpha=0$, we obtain that
\begin{equation*}
\begin{split}
\|(e^{\lambda_{1}2^{j}}a_{j}(u_{0}))_{j \in \mathbb{Z}}\|_{l^{0}_{\infty}}&=\sup_{j \in \mathbb{Z}}e^{\lambda_{1}2^{j}}\max_{2^{j}\leq t\leq 2^{j+1}}\|\partial_{t}u(t)\|_{L_{p}(\Omega)} \\
&\leq \sup_{j \in \mathbb{Z}}\left(e^{\lambda_{1}2^{j}}\max_{2^{j}\leq t\leq 2^{j+1}}Ce^{-\lambda_{1}t}\right)\|u_{0}\|_{X^{1}_{p}(\Omega)} \\
&=C\|u_{0}\|_{X^{1}_{p}(\Omega)},
\end{split}
\end{equation*}
where $C$ is a positive constant depending only on $n$, $\Omega$, $p$, $\kappa$, $\kappa_{s}$ and $\lambda_{1}$.
Similarly to the above estimate, since
\begin{equation*}
\|\partial_{t}u(t)\|_{L_{p}(\Omega)}\leq Ct^{-1}e^{-\lambda_{1}t}\|u_{0}\|_{L_{p}(\Omega)}
\end{equation*}
for any $t>0$, where $C$ is a positive constant depending only on $n$, $\Omega$, $p$, $\kappa$, $\kappa_{s}$ and $\lambda_{1}$, which is clear from (2.3) with $\alpha=1$, we have the following inequality:
\begin{equation*}
\begin{split}
\|(e^{\lambda_{1}2^{j}}a_{j}(u_{0}))_{j \in \mathbb{Z}}\|_{l^{1}_{\infty}}&=\sup_{j \in \mathbb{Z}}2^{j}e^{\lambda_{1}2^{j}}\max_{2^{j}\leq t\leq 2^{j+1}}\|\partial_{t}u(t)\|_{L_{p}(\Omega)} \\
&\leq \sup_{j \in \mathbb{Z}}\left(2^{j}e^{\lambda_{1}2^{j}}\max_{2^{j}\leq t\leq 2^{j+1}}Ct^{-1}e^{-\lambda_{1}t}\right)\|u_{0}\|_{L_{p}(\Omega)} \\
&=C\|u_{0}\|_{L_{p}(\Omega)},
\end{split}
\end{equation*}
where $C$ is a positive constant depending only on $n$, $\Omega$, $p$, $\kappa$, $\kappa_{s}$ and $\lambda_{1}$.
Therefore, we can conclude from $l^{s}_{q}=(l^{1}_{\infty},l^{0}_{\infty})_{s,q}$ for any $0<s<1$ that
\begin{equation}
\|(e^{\lambda_{1}2^{j}}a_{j}(u_{0}))_{j \in \mathbb{Z}}\|_{l^{s}_{q}}\leq C\|u_{0}\|_{(X^{1}_{p}(\Omega),L_{p}(\Omega))_{s,q}}
\end{equation}
for any $0<s<1$, where $C$ is a positive constant depending only on $n$, $\Omega$, $p$, $q$, $s$, $\kappa$, $\kappa_{s}$ and $\lambda_{1}$.
By $(X^{1}_{p}(\Omega),L_{p}(\Omega))_{s,q}=(L_{p}(\Omega),X^{1}_{p}(\Omega))_{1-s,q}$ and letting $s=1/q$, it is obvious from (3.6), (3.7) that
\begin{equation}
\|e^{(\lambda_{1}/2)t}\partial_{t}u\|_{L_{q}(\mathbb{R}_{+};L_{p}(\Omega))}\leq C\|u_{0}\|_{X_{p,q}(\Omega)},
\end{equation}
where $C$ is a positive constant depending only on $n$, $\Omega$, $p$, $q$, $\kappa$, $\kappa_{s}$ and $\lambda_{1}$.
Moreover, as for $\|e^{(\lambda_{1}/2)t}u\|_{L_{q}(\mathbb{R}_{+};W^{2}_{p}(\Omega))}$, it is derived from Theorem 3.1 with $f=-\partial_{t}u$, $g=0$ that
\begin{equation}
\|u(t)\|_{W^{2}_{p}(\Omega)}\leq C\|\partial_{t}u(t)\|_{L_{p}(\Omega)}
\end{equation}
for any $t>0$, where $C$ is a positive constant depending only on $n$, $\Omega$, $p$, $\kappa$ and $\kappa_{s}$.
We can easily see from (3.8), (3.9) that
\begin{equation}
\|e^{(\lambda_{1}/2)t}u\|_{L_{q}(\mathbb{R}_{+};W^{2}_{p}(\Omega))}\leq C\|u_{0}\|_{X_{p,q}(\Omega)},
\end{equation}
where $C$ is a positive constant depending only on $n$, $\Omega$, $p$, $q$, $\kappa$, $\kappa_{s}$ and $\lambda_{1}$.
Therefore, (3.5), (3.8), (3.10) clearly lead to (3.4).
\end{proof}

\subsection{$L_{p}$-$L_{q}$ estimates in $\mathbb{R}^{n}\times\mathbb{R}$ and in $\mathbb{R}^{n}_{+}\times\mathbb{R}$ with constant coefficients}
In the case where $\kappa$ is a positive constant, first of all, we discuss the following problem in $\mathbb{R}^{n}\times\mathbb{R}$:
\begin{equation}
\partial_{t}u-\kappa\Delta u+u=f \ \mathrm{in} \ \mathbb{R}^{n}\times\mathbb{R}.
\end{equation}
We can utilize the operator-valued Fourier multiplier theorem on $L_{p}$ to obtain the following lemma:
\begin{lemma}
Let $\kappa$ be a positive constant, $1<p<\infty$, $1<q<\infty$, $f \in L_{q,0}(\mathbb{R}_{+};L_{p}(\mathbb{R}^{n}))$.
Then $(3.11)$ has uniquely a solution $u \in W^{2,1}_{p,q,0}(\mathbb{R}^{n}\times\mathbb{R}_{+})$ satisfying
\begin{equation}
u(x,t)=\mathcal{L}^{-1}_{(\xi,\lambda_{2})}\left[\frac{\mathcal{L}_{(x,t)}[f(x,t)](\xi,\lambda)}{\lambda+\kappa|\xi|^{2}+1}\right](x,t),
\end{equation}
\begin{equation}
\sum^{2}_{k=0}\lambda^{k/2}_{1}\|e^{-\lambda_{1}t}u\|_{L_{q}(\mathbb{R};W^{2-k}_{p}(\mathbb{R}^{n}))}+\sum^{2}_{k=1}\|e^{-\lambda_{1}t}u\|_{H^{k/2}_{q}(\mathbb{R};W^{2-k}_{p}(\mathbb{R}^{n}))}\leq C\|e^{-\lambda_{1}t}f\|_{L_{q}(\mathbb{R};L_{p}(\mathbb{R}^{n}))}
\end{equation}
for any $\lambda_{1}\geq 0$, where $C$ is a positive constant depending only on $n$, $p$, $q$ and $\kappa$.
\end{lemma}
\begin{proof}
Since $C^{\infty}_{0}(G\times\mathbb{R}_{+})$ is dense in $L_{q,0}(\mathbb{R}_{+};L_{p}(G))$ for any open set $G$ in $\mathbb{R}^{n}$, we can assume that $f \in C^{\infty}_{0}(\mathbb{R}^{n}\times\mathbb{R}_{+})$.
By applying the Fourier-Laplace transform with respect to $(x,t) \in \mathbb{R}^{n}\times\mathbb{R}$ to (3.11), it follows that
\begin{equation*}
(\lambda+\kappa|\xi|^{2}+1)\mathcal{L}_{(x,t)}[u]=\mathcal{L}_{(x,t)}[f] \ \mathrm{in} \ \mathbb{R}^{n}\times\mathbb{R},
\end{equation*}
where $\xi \in \mathbb{R}^{n}$, $\lambda \in \mathbb{C}$, $\lambda=\lambda_{1}+\sqrt{-1}\lambda_{2}$, $\lambda_{1}, \lambda_{2} \in \mathbb{R}$.
Therefore, $u$ can be defined as in (3.12).
It is derived from (2.1), (2.2), (3.12) that we have the following formulas:
\begin{equation*}
P^{k,\alpha}_{\lambda_{1}}(x,\lambda_{2}):=\mathcal{F}^{-1}_{\xi}\left[\frac{\lambda^{k/2}_{2}(\sqrt{-1}\xi)^{\alpha}}{\lambda+\kappa|\xi|^{2}+1}\right](x,\lambda_{2}),
\end{equation*}
\begin{equation*}
Q^{k,\alpha}_{\lambda_{1}}(x,\lambda_{2}):=\mathcal{F}^{-1}_{\xi}\left[\frac{(1+\lambda^{2}_{2})^{k/4}(\sqrt{-1}\xi)^{\alpha}}{\lambda+\kappa|\xi|^{2}+1}\right](x,\lambda_{2}),
\end{equation*}
\begin{equation*}
(M^{k,\alpha}_{\lambda_{1}}(\lambda_{2})f)(x,\lambda_{2}):=\int_{\mathbb{R}^{n}}P^{k,\alpha}_{\lambda_{1}}(x-y,\lambda_{2})f(y)dy,
\end{equation*}
\begin{equation*}
(N^{k,\alpha}_{\lambda_{1}}(\lambda_{2})f)(x,\lambda_{2}):=\int_{\mathbb{R}^{n}}Q^{k,\alpha}_{\lambda_{1}}(x-y,\lambda_{2})f(y)dy,
\end{equation*}
\begin{equation}
\lambda^{k/2}_{1}\partial^{\alpha}_{x}(e^{-\lambda_{1}t}u)(x,t)=\mathcal{F}^{-1}_{\lambda_{2}}[(M^{k,\alpha}_{\lambda_{1}}(\lambda_{2})\mathcal{F}_{t}[e^{-\lambda_{1}t}f(x,t)])(x,\lambda_{2})](x,t),
\end{equation}
\begin{equation}
\langle \partial_{t} \rangle^{k/2}\partial^{\alpha}_{x}(e^{-\lambda_{1}t}u)(x,t)=\mathcal{F}^{-1}_{\lambda_{2}}[(N^{k,\alpha}_{\lambda_{1}}(\lambda_{2})\mathcal{F}_{t}[e^{-\lambda_{1}t}f(x,t)])(x,\lambda_{2})](x,t)
\end{equation}
for any $k \in \mathbb{Z}$, $k\geq 0$, $\alpha \in \mathbb{Z}^{n}$, $\alpha\geq 0$, $k+|\alpha|\leq 2$.
Set
\begin{equation*}
\mathcal{M}^{k,\alpha}_{\lambda_{1},0}=\{M^{k,\alpha}_{\lambda_{1}}(\lambda_{2})\}_{\lambda_{2} \in \mathbb{R}\setminus\{0\}},
\end{equation*}
\begin{equation*}
\mathcal{M}^{k,\alpha}_{\lambda_{1},1}=\{\lambda_{2}\partial_{\lambda_{2}}M^{k,\alpha}_{\lambda_{1}}(\lambda_{2})\}_{\lambda_{2} \in \mathbb{R}\setminus\{0\}},
\end{equation*}
\begin{equation*}
K^{k,\alpha}_{\lambda_{1}}=\mathcal{F}^{-1}_{\lambda_{2}}M^{k,\alpha}_{\lambda_{1}}\mathcal{F}_{t}, \ \mathcal{D}(K^{k,\alpha}_{\lambda_{1}})=C^{\infty}_{0}(\mathbb{R}\setminus\{0\};L_{2}(\mathbb{R}^{n})),
\end{equation*}
\begin{equation*}
\mathcal{N}^{k,\alpha}_{\lambda_{1},0}=\{N^{k,\alpha}_{\lambda_{1}}(\lambda_{2})\}_{\lambda_{2} \in \mathbb{R}\setminus\{0\}},
\end{equation*}
\begin{equation*}
\mathcal{N}^{k,\alpha}_{\lambda_{1},1}=\{\lambda_{2}\partial_{\lambda_{2}}N^{k,\alpha}_{\lambda_{1}}(\lambda_{2})\}_{\lambda_{2} \in \mathbb{R}\setminus\{0\}},
\end{equation*}
\begin{equation*}
L^{k,\alpha}_{\lambda_{1}}=\mathcal{F}^{-1}_{\lambda_{2}}N^{k,\alpha}_{\lambda_{1}}\mathcal{F}_{t}, \ \mathcal{D}(L^{k,\alpha}_{\lambda_{1}})=C^{\infty}_{0}(\mathbb{R}\setminus\{0\};L_{2}(\mathbb{R}^{n})).
\end{equation*}
We can easily see that
\begin{equation}
|\lambda+\kappa|\xi|^{2}+1|\geq c(|\lambda|+|\xi|^{2}+1)
\end{equation}
for any $\xi \in \mathbb{R}^{n}$, $\lambda \in \mathbb{C}$, $\lambda_{1}\geq 0$, where $c$ is a positive constant depending only on $\kappa$.
By applying the Plancherel theorem to $P^{k,\alpha}_{\lambda_{1}}$, it follows from (3.16) that
\begin{equation}
\|M^{k,\alpha}_{\lambda_{1}}(\lambda_{2})\|_{\mathcal{B}(L_{2}(\mathbb{R}^{n}))}\leq C
\end{equation}
for any $\lambda \in \mathbb{C}$, $\lambda_{1}\geq 0$, $k \in \mathbb{Z}$, $k\geq 0$, $\alpha \in \mathbb{Z}^{n}$, $\alpha\geq 0$, $k+|\alpha|\leq 2$, where $C$ is a positive constant depending only on $\kappa$.
We also obtain from (3.16) that
\begin{equation}
\left|\partial^{\gamma}_{\xi}\left(\frac{\lambda^{k/2}_{2}(\sqrt{-1}\xi)^{\alpha}}{\lambda+\kappa|\xi|^{2}+1}(\sqrt{-1}\xi)^{\beta}\right)\right|\leq C|\xi|^{|\beta|-|\gamma|}
\end{equation}
for any $\xi \in \mathbb{R}^{n}\setminus\{0\}$, $\lambda \in \mathbb{C}$, $\lambda_{1}\geq 0$, $k \in \mathbb{Z}$, $k\geq 0$, $\alpha \in \mathbb{Z}^{n}$, $\alpha\geq 0$, $k+|\alpha|\leq 2$, $\beta \in \mathbb{Z}^{n}$, $\beta\geq 0$, $\gamma \in \mathbb{Z}^{n}$, $\gamma\geq 0$, where $C$ is a positive constant depending only on $\kappa$, $\beta$ and $\gamma$.
It is derived from (3.18) and Lemma 2.4 with $m=n$, $\mu=1$, $g(\xi)=\mathcal{F}_{x}[P^{k,\alpha}_{\lambda_{1}}(x,\lambda_{2})](\xi)(\sqrt{-1}\xi)^{\beta}$ that
\begin{equation}
\sum_{\beta \in \mathbb{Z}^{n}, \beta\geq 0, |\beta|=1}|\partial^{\beta}_{x}P^{k,\alpha}_{\lambda_{1}}(x,\lambda_{2})|\leq C|x|^{-(n+1)}
\end{equation}
for any $x \in \mathbb{R}^{n}\setminus\{0\}$, $\lambda \in \mathbb{C}$, $\lambda_{1}\geq 0$, $k \in \mathbb{Z}$, $k\geq 0$, $\alpha \in \mathbb{Z}^{n}$, $\alpha\geq 0$, $k+|\alpha|\leq 2$, where $C$ is a positive constant depending only on $n$ and $\kappa$.
Therefore, from (3.17), (3.19) and Lemma 2.6, $\mathcal{M}^{k,\alpha}_{\lambda_{1},0}$ is $\mathcal{R}$-bounded on $\mathcal{B}(L_{p}(\mathbb{R}^{n}))$ for any $\lambda_{1}\geq 0$, $k \in \mathbb{Z}$, $k\geq 0$, $\alpha \in \mathbb{Z}^{n}$, $\alpha\geq 0$, $k+|\alpha|\leq 2$, $\mathcal{R}(\mathcal{M}^{k,\alpha}_{\lambda_{1},0})\leq C_{0}$, where $C_{0}$ is a positive constant depending only on $n$, $p$ and $\kappa$.
Similarly to $\mathcal{M}^{k,\alpha}_{\lambda_{1},0}$, we can conclude from Lemma 2.6 that $\mathcal{M}^{k,\alpha}_{\lambda_{1},1}$ is also $\mathcal{R}$-bounded on $\mathcal{B}(L_{p}(\mathbb{R}^{n}))$ for any $\lambda_{1}\geq 0$, $k \in \mathbb{Z}$, $k\geq 0$, $\alpha \in \mathbb{Z}^{n}$, $\alpha\geq 0$, $k+|\alpha|\leq 2$, $\mathcal{R}(\mathcal{M}^{k,\alpha}_{\lambda_{1},1})\leq C_{1}$, where $C_{1}$ is a positive constant depending only on $n$, $p$ and $\kappa$.
It follows from Lemma 2.5 that $K^{k,\alpha}_{\lambda_{1}}$ is extended to a bounded linear operator $K^{k,\alpha}_{\lambda_{1}} \in \mathcal{B}(L_{q}(\mathbb{R};L_{p}(\mathbb{R}^{n})))$ for any $\lambda_{1}\geq 0$, $k \in \mathbb{Z}$, $k\geq 0$, $\alpha \in \mathbb{Z}^{n}$, $\alpha\geq 0$, $k+|\alpha|\leq 2$, $\|K^{k,\alpha}_{\lambda_{1}}\|_{\mathcal{B}(L_{q}(\mathbb{R};L_{p}(\mathbb{R}^{n})))}\leq C$, where $C$ is a positive constant depending only on $n$, $p$, $q$ and $\kappa$.
The same argument as above shows that $\mathcal{N}^{k,\alpha}_{\lambda_{1},0}$ and $\mathcal{N}^{k,\alpha}_{\lambda_{1},1}$ are $\mathcal{R}$-bounded on $\mathcal{B}(L_{p}(\mathbb{R}^{n}))$, therefore, $L^{k,\alpha}_{\lambda_{1}}$ is extended to a bounded linear operator $L^{k,\alpha}_{\lambda_{1}} \in \mathcal{B}(L_{q}(\mathbb{R};L_{p}(\mathbb{R}^{n})))$ for any $\lambda_{1}\geq 0$, $k \in \mathbb{Z}$, $k\geq 0$, $\alpha \in \mathbb{Z}^{n}$, $\alpha\geq 0$, $k+|\alpha|\leq 2$, $\|L^{k,\alpha}_{\lambda_{1}}\|_{\mathcal{B}(L_{q}(\mathbb{R};L_{p}(\mathbb{R}^{n})))}\leq C$, where $C$ is a positive constant depending only on $n$, $p$, $q$ and $\kappa$.
It follows from (3.14), (3.15), $K^{k,\alpha}_{\lambda_{1}} \in \mathcal{B}(L_{q}(\mathbb{R};L_{p}(\mathbb{R}^{n})))$, $L^{k,\alpha}_{\lambda_{1}} \in \mathcal{B}(L_{q}(\mathbb{R};L_{p}(\mathbb{R}^{n})))$ that
\begin{equation}
\lambda^{k/2}_{1}\|\partial^{\alpha}_{x}(e^{-\lambda_{1}t}u)\|_{L_{q}(\mathbb{R};L_{p}(\mathbb{R}^{n}))}\leq C\|e^{-\lambda_{1}t}f\|_{L_{q}(\mathbb{R};L_{p}(\mathbb{R}^{n}))},
\end{equation}
\begin{equation}
\|\langle \partial_{t} \rangle^{k/2}\partial^{\alpha}_{x}(e^{-\lambda_{1}t}u)\|_{L_{q}(\mathbb{R};L_{p}(\mathbb{R}^{n}))}\leq C\|e^{-\lambda_{1}t}f\|_{L_{q}(\mathbb{R};L_{p}(\mathbb{R}^{n}))}
\end{equation}
for any $\lambda_{1}\geq 0$, $k \in \mathbb{Z}$, $k\geq 0$, $\alpha \in \mathbb{Z}^{n}$, $\alpha\geq 0$, $k+|\alpha|\leq 2$, where $C$ is a positive constant depending only on $n$, $p$, $q$ and $\kappa$.
It is clear from (3.20), (3.31) that (3.13) is established.
We remark that $e^{-\lambda_{1}t}\geq 1$ for any $\lambda_{1}\geq 0$, $t>0$ and that $f(t)=0$ for any $t<0$, and obtain from (3.13) that
\begin{equation*}
\begin{split}
\lambda_{1}\|u\|_{L_{q}(\mathbb{R}_{-};L_{p}(\mathbb{R}^{n}))}&\leq \lambda_{1}\|e^{-\lambda_{1}t}u\|_{L_{q}(\mathbb{R}_{-};L_{p}(\mathbb{R}^{n}))} \\
&\leq C\|e^{-\lambda_{1}t}f\|_{L_{q}(\mathbb{R};L_{p}(\mathbb{R}^{n}))} \\
&\leq C\|f\|_{L_{q}(\mathbb{R};L_{p}(\mathbb{R}^{n}))},
\end{split}
\end{equation*}
\begin{equation}
\|u\|_{L_{q}(\mathbb{R}_{-};L_{p}(\mathbb{R}^{n}))}\leq C\lambda^{-1}_{1}\|f\|_{L_{q}(\mathbb{R};L_{p}(\mathbb{R}^{n}))}
\end{equation}
for any $\lambda_{1}>0$, where $C$ is a positive constant depending only on $n$, $p$, $q$ and $\kappa$.
It is obvious from (3.22) with $\lambda_{1}\rightarrow \infty$ that $u(t)=0$ for any $t<0$, therefore, $u \in W^{2,1}_{p,q,0}(\mathbb{R}^{n}\times\mathbb{R}_{+})$.
\end{proof}
Second, we consider the following problem in $\mathbb{R}^{n}_{+}\times\mathbb{R}$:
\begin{equation}
\begin{split}
&\partial_{t}u-\kappa\Delta u+u=f & \mathrm{in} \ \mathbb{R}^{n}_{+}\times\mathbb{R}, \\
&\kappa\partial_{x_{n}}u|_{\partial\mathbb{R}^{n}_{+}}=h & \mathrm{on} \ \partial\mathbb{R}^{n}_{+}\times\mathbb{R}.
\end{split}
\end{equation}
Lemmas 2.6--2.8 and 3.1 admit that $L_{p}$-$L_{q}$ estimates for time-global solutions to (3.23) are established as follows:
\begin{lemma}
Let $\kappa$ be a positive constant, $1<p<\infty$, $1<q<\infty$, $f \in L_{q,0}(\mathbb{R}_{+};L_{p}(\mathbb{R}^{n}_{+}))$, $h \in H^{1,1/2}_{p,q,0}(\mathbb{R}^{n}_{+}\times\mathbb{R}_{+})$.
Then $(3.23)$ has uniquely a solution $u \in W^{2,1}_{p,q,0}(\mathbb{R}^{n}_{+}\times\mathbb{R}_{+})$ satisfying
\begin{equation}
\begin{split}
&\lambda_{1}\|e^{-\lambda_{1}t}u\|_{L_{q}(\mathbb{R};L_{p}(\mathbb{R}^{n}_{+}))}+\|e^{-\lambda_{1}t}u\|_{W^{2,1}_{p,q}(\mathbb{R}^{n}_{+}\times\mathbb{R})} \\
&\leq C(\|e^{-\lambda_{1}t}f\|_{L_{q}(\mathbb{R};L_{p}(\mathbb{R}^{n}_{+}))}+\lambda^{1/2}_{1}\|e^{-\lambda_{1}t}h\|_{L_{q}(\mathbb{R};L_{p}(\mathbb{R}^{n}_{+}))}+\|e^{-\lambda_{1}t}h\|_{H^{1,1/2}_{p,q,0}(\mathbb{R}^{n}_{+}\times\mathbb{R})})
\end{split}
\end{equation}
for any $\lambda_{1}\geq 0$, where $C$ is a positive constant depending only on $n$, $p$, $q$ and $\kappa$.
\end{lemma}
\begin{proof}
Let $\varphi_{o}$ be the odd extension of $\varphi$ to $\mathbb{R}$ defined as
\begin{equation*}
\varphi_{o}(x,t)=
\begin{cases}
\varphi(x,t) & \mathrm{if} \ x_{n}>0, \\
-\varphi(x',-x_{n},t) & \mathrm{if} \ x_{n}<0.
\end{cases}
\end{equation*}
Then it follows from Lemma 3.1 that
\begin{equation*}
\partial_{t}v-\kappa\Delta v+v=f_{o} \ \mathrm{in} \ \mathbb{R}^{n}\times\mathbb{R}
\end{equation*}
has uniquely a solution $v \in W^{2,1}_{p,q,0}(\mathbb{R}^{n}\times\mathbb{R}_{+})$ satisfying
\begin{equation}
v(x,t)=\mathcal{L}^{-1}_{(\xi,\lambda_{2})}\left[\frac{\mathcal{L}_{(x,t)}[f_{o}(x,t)](\xi,\lambda)}{\lambda+\kappa|\xi|^{2}+1}\right](x,t),
\end{equation}
\begin{equation}
\sum^{2}_{k=0}\lambda^{k/2}_{1}\|e^{-\lambda_{1}t}v\|_{L_{q}(\mathbb{R};W^{2-k}_{p}(\mathbb{R}^{n}))}+\sum^{2}_{k=1}\|e^{-\lambda_{1}t}v\|_{H^{k/2}_{q}(\mathbb{R};W^{2-k}_{p}(\mathbb{R}^{n}))}\leq C\|e^{-\lambda_{1}t}f_{o}\|_{L_{q}(\mathbb{R};L_{p}(\mathbb{R}^{n}))}
\end{equation}
for any $\lambda_{1}\geq 0$, where $C$ is a positive constant depending only on $n$, $p$, $q$ and $\kappa$.
Moreover, it is derived from (3.25) that $\partial_{x_{n}}v(x',0,t)=0$ for any $(x',t) \in \mathbb{R}^{n-1}\times\mathbb{R}$.
By letting $u=v+w$, it is clear from $\partial_{x_{n}}v|_{\partial\mathbb{R}^{n}_{+}}=0$ that $w$ must satisfy the following problem in $\mathbb{R}^{n}_{+}\times\mathbb{R}$:
\begin{equation}
\begin{split}
&\partial_{t}w-\kappa\Delta w+w=0 & \mathrm{in} \ \mathbb{R}^{n}_{+}\times\mathbb{R}, \\
&\kappa\partial_{x_{n}}w|_{\partial\mathbb{R}^{n}_{+}}=h & \mathrm{on} \ \partial\mathbb{R}^{n}_{+}\times\mathbb{R}.
\end{split}
\end{equation}
Therefore, it is essential for Lemma 3.2 to be proved that (3.27) has uniquely a solution $w \in W^{2,1}_{p,q,0}(\mathbb{R}^{n}_{+}\times\mathbb{R}_{+})$ satisfying
\begin{equation}
\begin{split}
\lambda_{1}\|e^{-\lambda_{1}t}w\|_{L_{q}(\mathbb{R};L_{p}(\mathbb{R}^{n}_{+}))}&+\|e^{-\lambda_{1}t}w\|_{W^{2,1}_{p,q}(\mathbb{R}^{n}_{+}\times\mathbb{R})} \\
&\leq C(\lambda^{1/2}_{1}\|e^{-\lambda_{1}t}h\|_{L_{q}(\mathbb{R};L_{p}(\mathbb{R}^{n}_{+}))}+\|e^{-\lambda_{1}t}h\|_{H^{1,1/2}_{p,q,0}(\mathbb{R}^{n}_{+}\times\mathbb{R})})
\end{split}
\end{equation}
for any $\lambda_{1}\geq 0$, where $C$ is a positive constant depending only on $n$, $p$, $q$ and $\kappa$.

We can utilize the Fourier-Laplace transform with respect to $(x',t) \in \mathbb{R}^{n-1}\times\mathbb{R}$ to rewrite (3.27) by the following problem:
\begin{equation*}
\begin{split}
&(\kappa|\xi'|^{2}+\lambda+1)\mathcal{L}_{(x',t)}[w]-\kappa\partial^{2}_{x_{n}}\mathcal{L}_{(x',t)}[w]=0 & \mathrm{in} \ \{x_{n}>0\}, \\
&\kappa\partial_{x_{n}}\mathcal{L}_{(x',t)}[w]|_{x_{n}=0}=\mathcal{L}_{(x',t)}[h] & \mathrm{on} \ \{x_{n}=0\}.
\end{split}
\end{equation*}
The inverse Fourier-Laplace transform with respect to $(\xi',\lambda_{2}) \in \mathbb{R}^{n-1}\times\mathbb{R}$ admits that $w$ is defined as
\begin{equation}
\begin{split}
&w(x,t)=-\mathcal{L}^{-1}_{(\xi',\lambda_{2})}[\kappa^{-1}A(\xi',\lambda)^{-1}\hat{h}(\xi',0,\lambda)e^{-A(\xi',\lambda)x_{n}}](x',t), \\
&\hat{h}(\xi',x_{n},\lambda)=\mathcal{L}_{(x',t)}[h(x,t)](\xi',x_{n},\lambda).
\end{split}
\end{equation}
It is clear from (3.29) that we have the following formula:
\begin{equation}
\begin{split}
w(x,t)&=-\int^{\infty}_{0}\partial_{y_{n}}\mathcal{L}^{-1}_{(\xi',\lambda_{2})}[\kappa^{-1}A(\xi',\lambda)^{-1}\hat{h}(\xi',y_{n},\lambda)e^{-A(\xi',\lambda)(x_{n}+y_{n})}](x',t)dy_{n} \\
&=\int^{\infty}_{0}\mathcal{L}^{-1}_{(\xi',\lambda_{2})}[\kappa^{-1}A(\xi',\lambda)^{-1}e^{-A(\xi',\lambda)(x_{n}+y_{n})}(A(\xi',\lambda)-\partial_{y_{n}})\hat{h}(\xi',y_{n},\lambda)](x',t)dy_{n}.
\end{split}
\end{equation}
Let 1 be decomposed into
\begin{equation*}
1=\frac{A(\xi',\lambda)^{2}}{A(\xi',\lambda)^{2}}=\frac{1}{A(\xi',\lambda)}\left\{\frac{\kappa^{-1}(\lambda+1)}{A(\xi',\lambda)}-\sum^{n-1}_{i=1}\frac{(\sqrt{-1}\xi_{i})^{2}}{A(\xi',\lambda)}\right\}.
\end{equation*}
Then it follows from (3.30) that
\begin{equation*}
\begin{split}
w(x,t)=&\int^{\infty}_{0}\mathcal{L}^{-1}_{(\xi',\lambda_{2})}[\kappa^{-1}A(\xi',\lambda)^{-1}e^{-A(\xi',\lambda)(x_{n}+y_{n})}A(\xi',\lambda)^{-1}\kappa^{-1}(\lambda+1)\hat{h}(\xi',y_{n},\lambda)](x',t)dy_{n} \\
&-\int^{\infty}_{0}\mathcal{L}^{-1}_{(\xi',\lambda_{2})}\left[\kappa^{-1}A(\xi',\lambda)^{-1}e^{-A(\xi',\lambda)(x_{n}+y_{n})}A(\xi',\lambda)^{-1}\sum^{n-1}_{i=1}(\sqrt{-1}\xi_{i})^{2}\hat{h}(\xi',y_{n},\lambda)\right](x',t)dy_{n} \\
&-\int^{\infty}_{0}\mathcal{L}^{-1}_{(\xi',\lambda_{2})}[\kappa^{-1}A(\xi',\lambda)^{-1}e^{-A(\xi',\lambda)(x_{n}+y_{n})}\partial_{y_{n}}\hat{h}(\xi',y_{n},\lambda)](x',t)dy_{n}.
\end{split}
\end{equation*}
Moreover, we can utilize the identity $\lambda+1=\lambda_{1}+(1+\sqrt{-1}\lambda_{2})$ $(\lambda_{1}, \lambda_{2} \in \mathbb{R})$ to obtain that
\begin{equation}
\begin{split}
w(x,t)=&\int^{\infty}_{0}\mathcal{L}^{-1}_{(\xi',\lambda_{2})}[A(\xi',\lambda)^{-1}e^{-A(\xi',\lambda)(x_{n}+y_{n})}\lambda_{1}^{1/2}A(\xi',\lambda)^{-1}\kappa^{-2}\lambda^{1/2}_{1}\hat{h}(\xi',y_{n},\lambda)](x',t)dy_{n} \\
&+\int^{\infty}_{0}\mathcal{L}^{-1}_{(\xi',\lambda_{2})}[A(\xi',\lambda)^{-1}e^{-A(\xi',\lambda)(x_{n}+y_{n})}(1+\sqrt{-1}\lambda_{2})(1+\lambda^{2}_{2})^{-1/4}A(\xi',\lambda)^{-1} \\
&\times\kappa^{-2}(1+\lambda^{2}_{2})^{1/4}\hat{h}(\xi',y_{n},\lambda)](x',t)dy_{n} \\
&-\int^{\infty}_{0}\mathcal{L}^{-1}_{(\xi',\lambda_{2})}\left[A(\xi',\lambda)^{-1}e^{-A(\xi',\lambda)(x_{n}+y_{n})}\sum^{n-1}_{i=1}(\sqrt{-1}\xi_{i})A(\xi',\lambda)^{-1}\kappa^{-1}(\sqrt{-1}\xi_{i})\hat{h}(\xi',y_{n},\lambda)\right](x',t)dy_{n} \\
&-\int^{\infty}_{0}\mathcal{L}^{-1}_{(\xi',\lambda_{2})}[A(\xi',\lambda)^{-1}e^{-A(\xi',\lambda)(x_{n}+y_{n})}\kappa^{-1}\partial_{y_{n}}\hat{h}(\xi',y_{n},\lambda)](x',t)dy_{n}.
\end{split}
\end{equation}
Since it is clear from the inverse Fourier-Laplace transform with respect to $(\xi',\lambda_{2}) \in \mathbb{R}^{n-1}\times\mathbb{R}$ that
\begin{equation*}
\mathcal{L}^{-1}_{(\xi',\lambda_{2})}[\kappa^{-2}\lambda^{1/2}_{1}\hat{h}(\xi',x_{n},\lambda)](x',t)=\kappa^{-2}\lambda^{1/2}_{1}h(x,t),
\end{equation*}
\begin{equation*}
\mathcal{L}^{-1}_{(\xi',\lambda_{2})}[\kappa^{-2}(1+\lambda^{2}_{2})^{1/4}\hat{h}(\xi',x_{n},\lambda)](x',t)=\kappa^{-2}e^{\lambda_{1}t}\langle \partial_{t} \rangle^{1/2}(e^{-\lambda_{1}t}h)(x,t),
\end{equation*}
\begin{equation*}
\mathcal{L}^{-1}_{(\xi',\lambda_{2})}\left[\kappa^{-1}\sum^{n-1}_{i=1}(\sqrt{-1}\xi_{i})\hat{h}(\xi',x_{n},\lambda)\right](x',t)=\kappa^{-1}\sum^{n-1}_{i=1}\partial_{x_{i}}h(x,t),
\end{equation*}
\begin{equation*}
\mathcal{L}^{-1}_{(\xi',\lambda_{2})}[\kappa^{-1}\partial_{x_{n}}\hat{h}(\xi',x_{n},\lambda)](x',t)=\kappa^{-1}\partial_{x_{n}}h(x,t),
\end{equation*}
we can apply Lemma 2.10 to $w$ and $\partial_{t}w$ by letting
\begin{equation*}
M_{1}(\xi',\lambda)=
\begin{cases}
(\lambda_{1}+1)\lambda_{1}^{1/2}A(\xi',\lambda)^{-1}, \\
(\lambda_{1}+1)(1+\sqrt{-1}\lambda_{2})(1+\lambda^{2}_{2})^{-1/4}A(\xi',\lambda)^{-1}, \\
(\lambda_{1}+1)\displaystyle\sum^{n-1}_{i=1}(\sqrt{-1}\xi_{i})A(\xi',\lambda)^{-1}, \\
\lambda_{1}+1 & \mathrm{if} \ w_{1}(x,t)=(\lambda_{1}+1)w(x,t), \\
\sqrt{-1}\lambda_{2}\lambda_{1}^{1/2}A(\xi',\lambda)^{-1}, \\
\sqrt{-1}\lambda_{2}(1+\sqrt{-1}\lambda_{2})(1+\lambda^{2}_{2})^{-1/4}A(\xi',\lambda)^{-1}, \\
\sqrt{-1}\lambda_{2}\displaystyle\sum^{n-1}_{i=1}(\sqrt{-1}\xi_{i})A(\xi',\lambda)^{-1}, \\
\sqrt{-1}\lambda_{2} & \mathrm{if} \ w_{1}(x,t)=e^{\lambda_{1}t}\partial_{t}(e^{-\lambda_{1}t}w(x,t)).
\end{cases}
\end{equation*}
Then we notice the identity $e^{-\lambda_{1}t}\partial_{t}w=\lambda_{1}e^{-\lambda_{1}t}w+\partial_{t}(e^{-\lambda_{1}t}w)$, and obtain from (3.31) that
\begin{equation}
\begin{split}
&\lambda_{1}\|e^{-\lambda_{1}t}w\|_{L_{q}(\mathbb{R};L_{p}(\mathbb{R}^{n}_{+}))}+\|e^{-\lambda_{1}t}w\|_{W^{1}_{q}(\mathbb{R};L_{p}(\mathbb{R}^{n}_{+}))} \\
&\leq C(\lambda^{1/2}_{1}\|e^{-\lambda_{1}t}h\|_{L_{q}(\mathbb{R};L_{p}(\mathbb{R}^{n}_{+}))}+\|\langle \partial_{t} \rangle^{1/2}(e^{-\lambda_{1}t}h)\|_{L_{q}(\mathbb{R};L_{p}(\mathbb{R}^{n}_{+}))}+\|\nabla(e^{-\lambda_{1}t}h)\|_{L_{q}(\mathbb{R};L_{p}(\mathbb{R}^{n}_{+}))})
\end{split}
\end{equation}
for any $\lambda_{1}\geq 0$, where $C$ is a positive constant depending only on $n$, $p$, $q$ and $\kappa$.
As for $\|e^{-\lambda_{1}t}w\|_{L_{q}(\mathbb{R};W^{2}_{p}(\mathbb{R}^{n}_{+}))}$, first, it can be easily seen from (3.30) that
\begin{equation}
\begin{split}
\partial^{\alpha'}_{x'}\partial_{x_{i}}w(x,t)=&\int^{\infty}_{0}\mathcal{L}^{-1}_{(\xi',\lambda_{2})}[A(\xi',\lambda)e^{-A(\xi',\lambda)(x_{n}+y_{n})}(\sqrt{-1}\xi_{i})A(\xi',\lambda)^{-1}\kappa^{-1}(\sqrt{-1}\xi')^{\alpha'}\hat{h}(\xi',y_{n},\lambda)](x',t)dy_{n} \\
&-\int^{\infty}_{0}\mathcal{L}^{-1}_{(\xi',\lambda_{2})}[A(\xi',\lambda)e^{-A(\xi',\lambda)(x_{n}+y_{n})}(\sqrt{-1}\xi')^{\alpha'}(\sqrt{-1}\xi_{i})A(\xi',\lambda)^{-2} \\
&\times\kappa^{-1}\partial_{y_{n}}\hat{h}(\xi',y_{n},\lambda)](x',t)dy_{n}
\end{split}
\end{equation}
for any $\alpha' \in \mathbb{Z}^{n-1}$, $\alpha'\geq 0$, $|\alpha'|\leq 1$, $i=1, \cdots, n-1$.
Since it is clear from the inverse Fourier-Laplace transform with respect to $(\xi',\lambda_{2}) \in \mathbb{R}^{n-1}\times\mathbb{R}$ that
\begin{equation*}
\mathcal{L}^{-1}_{(\xi',\lambda_{2})}[\kappa^{-1}(\sqrt{-1}\xi)^{\alpha'}\hat{h}(\xi',x_{n},\lambda)](x',t)=\kappa^{-1}\partial^{\alpha'}_{x'}h(x,t),
\end{equation*}
\begin{equation*}
\mathcal{L}^{-1}_{(\xi',\lambda_{2})}[\kappa^{-1}\partial_{x_{n}}\hat{h}(\xi',x_{n},\lambda)](x',t)=\kappa^{-1}\partial_{x_{n}}h(x,t),
\end{equation*}
we can apply Lemma 2.11 to $\partial^{\alpha'}_{x'}\partial_{x_{i}}w$ by letting
\begin{equation*}
M_{2}(\xi',\lambda)=
\begin{cases}
(\sqrt{-1}\xi_{i})A(\xi',\lambda)^{-1} & \mathrm{if} \ \hat{g}(\xi',x_{n},\lambda)=(\sqrt{-1}\xi')^{\alpha'}\hat{h}(\xi',x_{n},\lambda), \\
(\sqrt{-1}\xi')^{\alpha'}(\sqrt{-1}\xi_{i})A(\xi',\lambda)^{-2} & \mathrm{if} \ \hat{g}(\xi',x_{n},\lambda)=\partial_{x_{n}}\hat{h}(\xi',x_{n},\lambda).
\end{cases}
\end{equation*}
Then we obtain from (3.33) that
\begin{equation}
\|\partial^{\alpha'}_{x'}\partial_{x_{i}}(e^{-\lambda_{1}t}w)\|_{L_{q}(\mathbb{R};L_{p}(\mathbb{R}^{n}_{+}))}\leq C(\|\partial^{\alpha'}_{x'}(e^{-\lambda_{1}t}h)\|_{L_{q}(\mathbb{R};L_{p}(\mathbb{R}^{n}_{+}))}+\|\partial_{x_{n}}(e^{-\lambda_{1}t}h)\|_{L_{q}(\mathbb{R};L_{p}(\mathbb{R}^{n}_{+}))})
\end{equation}
for any $\lambda_{1}\geq 0$, $\alpha' \in \mathbb{Z}^{n-1}$, $\alpha'\geq 0$, $|\alpha'|\leq 1$, $i=1, \cdots, n-1$, where $C$ is a positive constant depending only on $n$, $p$, $q$ and $\kappa$.
Second, it follows from (3.30) that
\begin{equation*}
\partial_{x_{n}}w(x,t)=\mathcal{L}^{-1}_{(\xi',\lambda_{2})}[\kappa^{-1}e^{-A(\xi',\lambda)x_{n}}\hat{h}(\xi',0,\lambda)](x',t),
\end{equation*}
\begin{equation}
\begin{split}
\partial^{\beta'}_{x'}\partial_{x_{n}}w(x,t)=&\int^{\infty}_{0}\mathcal{L}^{-1}_{(\xi',\lambda_{2})}[A(\xi',\lambda)e^{-A(\xi',\lambda)(x_{n}+y_{n})}\kappa^{-1}(\sqrt{-1}\xi')^{\beta'}\hat{h}(\xi',y_{n},\lambda)](x',t)dy_{n} \\
&-\int^{\infty}_{0}\mathcal{L}^{-1}_{(\xi',\lambda_{2})}[e^{-A(\xi',\lambda)(x_{n}+y_{n})}(\sqrt{-1}\xi')^{\beta'}\kappa^{-1}\partial_{y_{n}}\hat{h}(\xi',y_{n},\lambda)](x',t)dy_{n}
\end{split}
\end{equation}
for any $\beta' \in \mathbb{Z}^{n-1}$, $\beta'\geq 0$, $|\beta'|\leq 1$.
By letting
\begin{equation*}
M_{2}(\xi',\lambda)=
\begin{cases}
1 & \mathrm{if} \ \hat{g}(\xi',x_{n},\lambda)=(\sqrt{-1}\xi')^{\beta'}\hat{h}(\xi',x_{n},\lambda), \\
(\sqrt{-1}\xi')^{\beta'}A(\xi',\lambda)^{-1} & \mathrm{if} \ \hat{g}(\xi',x_{n},\lambda)=\partial_{x_{n}}\hat{h}(\xi',x_{n},\lambda),
\end{cases}
\end{equation*}
we can utilize Lemma 2.11 to obtain from (3.35) that
\begin{equation}
\|\partial^{\beta'}_{x'}\partial_{x_{n}}(e^{-\lambda_{1}t}w)\|_{L_{q}(\mathbb{R};L_{p}(\mathbb{R}^{n}_{+}))}\leq C(\|\partial^{\beta'}_{x'}(e^{-\lambda_{1}t}h)\|_{L_{q}(\mathbb{R};L_{p}(\mathbb{R}^{n}_{+}))}+\|\partial_{x_{n}}(e^{-\lambda_{1}t}h)\|_{L_{q}(\mathbb{R};L_{p}(\mathbb{R}^{n}_{+}))})
\end{equation}
for any $\lambda_{1}\geq 0$, $\beta' \in \mathbb{Z}^{n-1}$, $\beta'\geq 0$, $|\beta'|\leq 1$, where $C$ is a positive constant depending only on $n$, $p$, $q$ and $\kappa$.
Finally, we notice the equation
\begin{equation}
\begin{split}
\partial^{2}_{x_{n}}w(x,t)&=-\mathcal{L}^{-1}_{(\xi',\lambda_{2})}[\kappa^{-1}A(\xi',\lambda)e^{-A(\xi',\lambda)x_{n}}\hat{h}(\xi',0,\lambda)](x',t) \\
&=\sum^{n-1}_{i=1}\partial^{2}_{x_{i}}w(x,t)-\kappa^{-1}\partial_{t}w(x,t)-\kappa^{-1}w(x,t)
\end{split}
\end{equation}
and the identity $e^{-\lambda_{1}t}\partial_{t}w=\lambda_{1}e^{-\lambda_{1}t}w+\partial_{t}(e^{-\lambda_{1}t}w)$, and obtain from (3.32), (3.34) that
\begin{equation}
\begin{split}
&\|\partial^{2}_{x_{n}}(e^{-\lambda_{1}t}w)\|_{L_{q}(\mathbb{R};L_{p}(\mathbb{R}^{n}_{+}))} \\
&\leq C(\lambda^{1/2}_{1}\|e^{-\lambda_{1}t}h\|_{L_{q}(\mathbb{R};L_{p}(\mathbb{R}^{n}_{+}))}+\|\langle \partial_{t} \rangle^{1/2}(e^{-\lambda_{1}t}h)\|_{L_{q}(\mathbb{R};L_{p}(\mathbb{R}^{n}_{+}))}+\|\nabla(e^{-\lambda_{1}t}h)\|_{L_{q}(\mathbb{R};L_{p}(\mathbb{R}^{n}_{+}))})
\end{split}
\end{equation}
for any $\lambda_{1}\geq 0$, where $C$ is a positive constant depending only on $n$, $p$, $q$ and $\kappa$.
Therefore, it is clear that (3.28) is established by (3.32), (3.34), (3.36), (3.38).

We remark that $e^{-\lambda_{1}t}\geq 1$ for any $\lambda_{1}\geq 0$, $t>0$ and that $h(t)=0$ for any $t<0$, and obtain from (3.28) that
\begin{equation*}
\begin{split}
\lambda_{1}\|w\|_{L_{q}(\mathbb{R}_{-};L_{p}(\mathbb{R}^{n}_{+}))}&\leq \lambda_{1}\|e^{-\lambda_{1}t}w\|_{L_{q}(\mathbb{R}_{-};L_{p}(\mathbb{R}^{n}_{+}))} \\
&\leq C(\lambda^{1/2}_{1}\|e^{-\lambda_{1}t}h\|_{L_{q}(\mathbb{R};L_{p}(\mathbb{R}^{n}_{+}))}+\|e^{-\lambda_{1}t}h\|_{H^{1,1/2}_{p,q,0}(\mathbb{R}^{n}_{+}\times\mathbb{R})}) \\
&\leq C(\lambda^{1/2}_{1}\|h\|_{L_{q}(\mathbb{R};L_{p}(\mathbb{R}^{n}_{+}))}+\|h\|_{H^{1,1/2}_{p,q,0}(\mathbb{R}^{n}_{+}\times\mathbb{R})}),
\end{split}
\end{equation*}
\begin{equation}
\|w\|_{L_{q}(\mathbb{R}_{-};L_{p}(\mathbb{R}^{n}_{+}))}\leq C(\lambda^{-1/2}_{1}\|h\|_{L_{q}(\mathbb{R};L_{p}(\mathbb{R}^{n}_{+}))}+\lambda^{-1}_{1}\|h\|_{H^{1,1/2}_{p,q,0}(\mathbb{R}^{n}_{+}\times\mathbb{R})})
\end{equation}
for any $\lambda_{1}>0$, where $C$ is a positive constant depending only on $n$, $p$, $q$ and $\kappa$.
It is obvious from (3.39) with $\lambda_{1}\rightarrow \infty$ that $w(t)=0$ for any $t<0$, therefore, $w \in W^{2,1}_{p,q,0}(\mathbb{R}^{n}_{+}\times\mathbb{R}_{+})$.
\end{proof}

\subsection{$L_{p}$-$L_{q}$ estimates in $\Omega\times\mathbb{R}_{+}$ with $u_{0}=0$}
First, we discuss (1.1) with a positive constant $\kappa$, $u_{0}=0$ and $g=0$, that is, the following initial-boundary value problem in $\Omega\times\mathbb{R}_{+}$:
\begin{equation}
\begin{split}
&\partial_{t}u-\kappa\Delta u=f & \mathrm{in} \ \Omega\times\mathbb{R}_{+}, \\
&u|_{t=0}=0 & \mathrm{in} \ \Omega, \\
&\kappa\partial_{\nu}u+\kappa_{s}u|_{\partial\Omega}=0 & \mathrm{on} \ \partial\Omega\times\mathbb{R}_{+}.
\end{split}
\end{equation}
Lemma 3.2 implies that $L_{p}$-$L_{q}$ estimates for time-global solutions to (3.40) are established as follows:
\begin{lemma}
Let $\kappa$ be a positive constant, $\kappa_{s} \in C^{1}(\partial\Omega)$ satisfy $\kappa_{s}>0$ on $\partial\Omega$, $1<p<\infty$, $1<q<\infty$, $f \in C^{\infty}_{0}(\mathbb{R}_{+};L_{p}(\Omega))$.
Then $(3.40)$ has uniquely a solution $u \in W^{2,1}_{p,q,0}(\Omega\times\mathbb{R}_{+})$ satisfying
\begin{equation*}
u(t)=\int^{t}_{0}e^{-(t-s)A_{p}}f(s)ds,
\end{equation*}
\begin{equation}
\|u\|_{W^{2,1}_{p,q}(\Omega\times\mathbb{R}_{+})}\leq C\|f\|_{L_{q}(\mathbb{R}_{+};L_{p}(\Omega))},
\end{equation}
where $C$ is a positive constant depending only on $n$, $\Omega$, $p$, $q$, $\kappa$ and $\kappa_{s}$.
\end{lemma}
\begin{proof}
It is essential for Lemma 3.3 to be proved that $u$ satisfies (3.41).
First, we obtain the following inequality:
\begin{equation}
\|u\|_{L_{q}(\mathbb{R}_{+};W^{1}_{p}(\Omega))}\leq C\|f\|_{L_{q}(\mathbb{R}_{+};L_{p}(\Omega))},
\end{equation}
where $C$ is a positive constant depending only on $n$, $\Omega$, $p$, $\kappa$ and $\kappa_{s}$.
It is clear from (2.3), (2.4) with $\alpha=1/2$ that
\begin{equation*}
\|u(t)\|_{W^{1}_{p}(\Omega)}\leq C\int^{t}_{0}(t-s)^{-1/2}e^{-\lambda_{1}(t-s)}\|f(s)\|_{L_{p}(\Omega)}ds
\end{equation*}
for any $t>0$, $0<\lambda_{1}<\Lambda_{1}$, where $C$ is a positive constant depending only on $n$, $\Omega$, $p$, $\kappa$, $\kappa_{s}$ and $\lambda_{1}$ and that
\begin{equation*}
\int^{t}_{0}(t-s)^{-1/2}e^{-\lambda_{1}(t-s)}\|f(s)\|_{L_{p}(\Omega)}ds=\int^{t}_{0}\|f(t-s)\|_{L_{p}(\Omega)}s^{-1/2}e^{-\lambda_{1}s}ds.
\end{equation*}
It follows from the H\"{o}lder inequality that
\begin{equation*}
\int^{t}_{0}\|f(t-s)\|_{L_{p}(\Omega)}s^{-1/2}e^{-\lambda_{1}s}ds\leq \left(\int^{t}_{0}\|f(t-s)\|^{q}_{L_{p}(\Omega)}s^{-1/2}e^{-\lambda_{1}s}ds\right)^{1/q}\left(\int^{t}_{0}s^{-1/2}e^{-\lambda_{1}s}ds\right)^{1/q^{*}}
\end{equation*}
for any $t>0$, $0<\lambda_{1}<\Lambda_{1}$, where $1<q^{*}<\infty$ is the dual exponent to $q$.
We interchange the order of integration with respect to $t$, and obtain that
\begin{equation*}
\int^{\infty}_{0}\left(\int^{t}_{0}\|f(t-s)\|_{L_{p}(\Omega)}s^{-1/2}e^{-\lambda_{1}s}ds\right)^{q}dt\leq \|f\|^{q}_{L_{q}(\mathbb{R}_{+};L_{p}(\Omega))}\left(\int^{\infty}_{0}s^{-1/2}e^{-\lambda_{1}s}ds\right)^{1+q/q^{*}},
\end{equation*}
\begin{equation*}
\|u\|_{L_{q}(\mathbb{R}_{+};W^{1}_{p}(\Omega))}\leq C\|f\|_{L_{q}(\mathbb{R}_{+};L_{p}(\Omega))}\int^{\infty}_{0}s^{-1/2}e^{-\lambda_{1}s}ds=C\left(\frac{\pi}{\lambda_{1}}\right)^{1/2}\|f\|_{L_{q}(\mathbb{R}_{+};L_{p}(\Omega))}
\end{equation*}
for any $0<\lambda_{1}<\Lambda_{1}$, where $C$ is a positive constant depending only on $n$, $\Omega$, $p$, $\kappa$, $\kappa_{s}$ and $\lambda_{1}$.
Therefore, it is sufficient for Lemma 3.3 to be proved that
\begin{equation}
\|u\|_{W^{2,1}_{p,q}(\Omega\times\mathbb{R}_{+})}\leq C(\|f\|_{L_{q}(\mathbb{R}_{+};L_{p}(\Omega))}+\|u\|_{L_{q}(\mathbb{R}_{+};W^{1}_{p}(\Omega))}),
\end{equation}
where $C$ is a positive constant depending only on $n$, $\Omega$, $p$, $q$, $\kappa$ and $\kappa_{s}$.

Second, we give $L_{p}$-$L_{q}$ interior estimates for time-global solutions to (3.40).
Set $\Omega_{\delta}=\{x \in \Omega \ ; \ \mathrm{dist}(x,\partial\Omega)>\delta\}$ for any $\delta>0$.
Then $\varphi \in C^{\infty}_{0}(\mathbb{R}^{n})$ can be defined as
\begin{equation*}
\begin{cases}
\varphi(x)=1 & \mathrm{if} \ x \in \Omega_{\delta}, \\
0<\varphi(x)<1 & \mathrm{if} \ x \in \Omega_{\delta/2}\setminus\overline{\Omega_{\delta}}, \\
\varphi(x)=0 & \mathrm{if} \ x \in \mathbb{R}^{n}\setminus\overline{\Omega_{\delta/2}}.
\end{cases}
\end{equation*}
It is clear from (3.40) that $\varphi u$ satisfies the following problem in $\mathbb{R}^{n}\times\mathbb{R}_{+}$:
\begin{equation}
\partial_{t}(\varphi u)-\kappa\Delta(\varphi u)+(\varphi u)=f_{\delta}+(\varphi u) \ \mathrm{in} \ \mathbb{R}^{n}\times\mathbb{R}_{+},
\end{equation}
where $f_{\delta}=(\varphi f)-\kappa(\Delta\varphi)u-2\kappa\nabla\varphi\cdot\nabla u$.
Since it follows from $f \in C^{\infty}_{0}(\mathbb{R}_{+};L_{p}(\Omega))$ that $f(t)=0$, $u(t)=0$ for any $t<0$, we can apply Lemma 3.1 with $\lambda_{1}=0$ to (3.44), and obtain that
\begin{equation}
\|\varphi u\|_{L_{q}(\mathbb{R}_{+};W^{2}_{p}(\mathbb{R}^{n}))}+\sum^{2}_{k=1}\|\varphi u\|_{H^{k/2}_{q}(\mathbb{R}_{+};W^{2-k}_{p}(\mathbb{R}^{n}))}\leq C\|f_{\delta}\|_{L_{q}(\mathbb{R}_{+};L_{p}(\mathbb{R}^{n}))},
\end{equation}
where $C$ is a positive constant depending only on $n$, $p$, $q$ and $\kappa$.
Moreover, it is obvious from the definition of $\varphi$ and $f_{\delta}$ that
\begin{equation*}
\|\varphi u\|_{L_{q}(\mathbb{R}_{+};L_{p}(\mathbb{R}^{n}))}\geq \|u\|_{L_{q}(\mathbb{R}_{+};L_{p}(\Omega_{\delta}))},
\end{equation*}
\begin{equation*}
\|f_{\delta}\|_{L_{q}(\mathbb{R}_{+};L_{p}(\mathbb{R}^{n}))}\leq \|f\|_{L_{q}(\mathbb{R}_{+};L_{p}(\Omega))}+C_{\delta}\|u\|_{L_{q}(\mathbb{R}_{+};W^{1}_{p}(\Omega))},
\end{equation*}
where $C_{\delta}$ is a positive constant depending only on $n$, $\Omega$, $p$, $q$, $\kappa$ and $\delta$.
Therefore, it follows from (3.45) that
\begin{equation}
\|u\|_{L_{q}(\mathbb{R}_{+};W^{2}_{p}(\Omega_{\delta}))}+\sum^{2}_{k=1}\|u\|_{H^{k/2}_{q}(\mathbb{R}_{+};W^{2-k}_{p}(\Omega_{\delta}))}\leq C_{\delta}(\|f\|_{L_{q}(\mathbb{R}_{+};L_{p}(\Omega))}+\|u\|_{L_{q}(\mathbb{R}_{+};W^{1}_{p}(\Omega))}),
\end{equation}
where $C_{\delta}$ is a positive constant depending only on $n$, $\Omega$, $p$, $q$, $\kappa$ and $\delta$.

Finally, we establish $L_{p}$-$L_{q}$ estimates near the boundary for time-global solutions to (3.40).
Set $B_{\delta}(x_{0})=\{x \in \mathbb{R}^{n} \ ; \ |x-x_{0}|<\delta\}$ for any $x_{0} \in \partial\Omega$, $\delta>0$.
Then $\varphi \in C^{\infty}_{0}(\mathbb{R}^{n})$ can be defined as
\begin{equation*}
\begin{cases}
\varphi(x)=1 & \mathrm{if} \ x \in B_{\delta}(x_{0}), \\
0<\varphi(x)<1 & \mathrm{if} \ x \in B_{2\delta}(x_{0})\setminus\overline{B_{\delta}(x_{0})}, \\
\varphi(x)=0 & \mathrm{if} \ x \in \mathbb{R}^{n}\setminus\overline{B_{2\delta}(x_{0})}.
\end{cases}
\end{equation*}
It is clear from (3.40) that $\varphi u$ satisfies the following problem in $\Omega\times\mathbb{R}_{+}$:
\begin{equation}
\begin{split}
&\partial_{t}(\varphi u)-\kappa\Delta(\varphi u)+(\varphi u)=f_{\delta}+(\varphi u) & \mathrm{in} \ \Omega\times\mathbb{R}_{+}, \\
&(\varphi u)|_{t=0}=0 & \mathrm{in} \ \Omega, \\
&\kappa\partial_{\nu}(\varphi u)+\kappa_{s}(\varphi u)|_{\partial\Omega}=f_{b,\delta} & \mathrm{on} \ \partial\Omega\times\mathbb{R}_{+},
\end{split}
\end{equation}
where $f_{\delta}=(\varphi f)-\kappa(\Delta\varphi)u-2\kappa\nabla\varphi\cdot\nabla u$, $f_{b,\delta}=\kappa(\partial_{\nu}\varphi)u$ and that there exists the orthogonal matrix $O$ such that $O^{T}\nu(x_{0})=(0,\cdots,0,-1)^{T}$.
Set $x=x_{0}+Oy$, $\widetilde{\Omega}=\{O^{T}(x-x_{0}) \ ; \ x \in \Omega\}$, $v(y,t)=(\varphi u)(x,t)$, $\tilde{\nu}(y)=O^{T}\nu(x)$.
Then it follows from (3.47) that $v$ satisfies the following problem in $\widetilde{\Omega}\times\mathbb{R}_{+}$:
\begin{equation}
\begin{split}
&\partial_{t}v-\kappa\Delta v+v=g_{\delta} & \mathrm{in} \ \widetilde{\Omega}\times\mathbb{R}_{+}, \\
&v|_{t=0}=0 & \mathrm{in} \ \widetilde{\Omega}, \\
&\kappa\partial_{\tilde{\nu}}v+\kappa_{s}v|_{\partial\widetilde{\Omega}}=g_{b,\delta} & \mathrm{on} \ \partial\widetilde{\Omega}\times\mathbb{R}_{+},
\end{split}
\end{equation}
where $g_{\delta}(y,t)=f_{\delta}(x,t)+v(y,t)$, $g_{b,\delta}(y,t)=\kappa(\partial_{\tilde{\nu}}\varphi)(y)u(x,t)$.
Let $\delta$ be a small positive constant satisfying $\mathrm{supp}v\subset \widetilde{B}_{\varepsilon_{0}}(0)$ for any small positive constant $\varepsilon_{0}$, where $\widetilde{B}_{\varepsilon_{0}}(0)=\{y \in \mathbb{R}^{n} \ ; \ |y|<\varepsilon_{0}\}$.
Then there exist a positive constant $\varepsilon_{1}>\varepsilon_{0}$ and $\chi \in C^{2}(\widetilde{B}'_{\varepsilon_{1}}(0))$ such that
\begin{equation*}
\widetilde{B}_{\varepsilon_{0}}(0)\cap\widetilde{\Omega}=\{y=(y',y_{n}) \in \mathbb{R}^{n} \ ; \ y_{n}>\chi(y'), \ |y'|<\varepsilon_{1}\},
\end{equation*}
\begin{equation*}
\widetilde{B}_{\varepsilon_{0}}(0)\cap\partial\widetilde{\Omega}=\{y=(y',y_{n}) \in \mathbb{R}^{n} \ ; \ y_{n}=\chi(y'), \ |y'|<\varepsilon_{1}\},
\end{equation*}
where $\widetilde{B}'_{\varepsilon_{1}}(0)=\{y' \in \mathbb{R}^{n-1} \ ; \ |y'|<\varepsilon_{1}\}$.
Moreover, it follows from $\tilde{\nu}(0)=(0,\cdots,0,-1)^{T}$ that $\chi(0)=0$, $\nabla'\chi(0)=0$, $\tilde{\nu}(y')=(\nabla'\chi(y'),-1)^{T}/\sqrt{1+|\nabla'\chi(y')|^{2}}$.
Let $\delta_{0}\geq\delta$ be a small positive constant satisfying $0<\varepsilon_{0}<\varepsilon_{1}/2$, $\psi \in C^{\infty}_{0}(\mathbb{R}^{n-1})$ be defined as
\begin{equation*}
\begin{cases}
\psi(y')=1 & \mathrm{if} \ |y'|\leq 1, \\
0<\psi(y')<1 & \mathrm{if} \ 1<|y'|<2, \\
\psi(y')=0 & \mathrm{if} \ |y'|\geq 2,
\end{cases}
\end{equation*}
$\omega(y')=\psi(y'/\varepsilon_{0})\chi(y')$, $H^{n}_{\omega}=\{y=(y',y_{n}) \in \mathbb{R}^{n} \ ; \ y_{n}>\omega(y'), \ y' \in \mathbb{R}^{n-1}\}$, $\partial H^{n}_{\omega}=\{y=(y',y_{n}) \in \mathbb{R}^{n} \ ; \ y_{n}=\omega(y'), \ y' \in \mathbb{R}^{n-1}\}$, $\nu_{\omega} \in C^{2}(\partial H^{n}_{\omega})$ be the outward unit normal vector on $\partial H^{n}_{\omega}$.
Then it follows from (3.48) that $v$ satisfies the following problem in $H^{n}_{\omega}\times\mathbb{R}_{+}$:
\begin{equation}
\begin{split}
&\partial_{t}v-\kappa\Delta v+v=g_{\delta} & \mathrm{in} \ H^{n}_{\omega}\times\mathbb{R}_{+}, \\
&v|_{t=0}=0 & \mathrm{in} \ H^{n}_{\omega}, \\
&\kappa\partial_{\nu_{\omega}}v+\kappa_{s}v|_{\partial H_{\omega}}=g_{b,\delta} & \mathrm{on} \ \partial H^{n}_{\omega}\times\mathbb{R}_{+}.
\end{split}
\end{equation}
Moreover, it is clear from the definition of $\omega$ and $\nu_{\omega}$ that $\omega(0)=0$, $\nabla'\omega(0)=0$, $\nu_{\omega}(y')=(\nabla'\omega(y'),-1)^{T}/\sqrt{1+|\nabla'\omega(y')|^{2}}$.
Since $\chi(0)=0$, $\nabla'\chi(0)=0$, we can utilize the following formula:
\begin{equation*}
\nabla'\omega(y')=\frac{1}{\varepsilon_{0}}\nabla'\psi\left(\frac{y'}{\varepsilon_{0}}\right)\chi(y')+\psi\left(\frac{y'}{\varepsilon_{0}}\right)\nabla'\chi(y')
\end{equation*}
and the following Taylor formula:
\begin{equation*}
\chi(y')=\frac{1}{2}\sum_{\alpha' \in \mathbb{Z}^{n-1}, \alpha'\geq 0, |\alpha'|=2}\partial^{\alpha'}_{y'}\chi(y'_{0})y'^{\alpha'}, \ \exists y'_{0} \in \widetilde{B}'_{\varepsilon_{1}}(0),
\end{equation*}
\begin{equation*}
\nabla'\chi(y')=\nabla'^{2}\chi(y'_{1})y'^{T}, \ \exists y'_{1} \in \widetilde{B}'_{\varepsilon_{1}}(0)
\end{equation*}
to obtain that
\begin{equation*}
\begin{split}
|\nabla'\omega(y')|&\leq \frac{n}{2}\left(\frac{|y'|}{\varepsilon_{0}}\right)^{2}\left|\nabla'\psi\left(\frac{y'}{\varepsilon_{0}}\right)\right||\nabla'^{2}\chi(y'_{0})|\varepsilon_{0}+\left(\frac{|y'|}{\varepsilon_{0}}\right)\left|\psi\left(\frac{y'}{\varepsilon_{0}}\right)\right||\nabla'^{2}\chi(y'_{1})|\varepsilon_{0} \\
&\leq \frac{n}{2}\max_{|y'|\leq 2}(|y'|^{2}|\nabla'\psi(y')|)|\nabla'^{2}\chi(y'_{0})|\varepsilon_{0}+\max_{|y'|\leq 2}(|y'||\psi(y')|)|\nabla'^{2}\chi(y'_{1})|\varepsilon_{0} \\
&\leq C\max_{|y'|\leq\varepsilon_{1}}|\nabla'^{2}\chi(y')|\varepsilon_{0},
\end{split}
\end{equation*}
\begin{equation}
\|\nabla'\omega\|_{C_{b}(\mathbb{R}^{n-1})}\leq C\max_{|y'|\leq\varepsilon_{1}}|\nabla'^{2}\chi(y')|\varepsilon_{0},
\end{equation}
where $C$ is a positive constant depending only on $n$.
Set
\begin{equation*}
M_{k}=\sum_{\alpha' \in \mathbb{Z}^{n-1}, \alpha'\geq 0, |\alpha'|=k}\|\partial^{\alpha'}\omega\|_{C_{b}(\mathbb{R}^{n-1})}
\end{equation*}
for any $k=0, 1, 2$.
Then
\begin{equation}
M_{1}\leq C\max_{|y'|\leq\varepsilon_{1}}|\nabla'^{2}\chi(y')|\varepsilon_{0},
\end{equation}
where $C$ is a positive constant depending only on $n$.
Set $(z',z_{n})=(y',y_{n}-\omega(y'))$, $w(z,t)=v(y,t)$, $h_{\delta}(z,t)=g_{\delta}(y,t)$.
Then it follows from (3.49) and the following formulas:
\begin{equation*}
\partial_{y_{i}}=
\begin{cases}
\partial_{z_{i}}-(\partial_{z_{i}}\omega)\partial_{z_{n}} & \mathrm{if} \ i=1, \cdots, n-1, \\
\partial_{z_{n}} & \mathrm{if} \ i=n,
\end{cases}
\end{equation*}
\begin{equation*}
\partial^{2}_{y_{i}}=
\begin{cases}
\partial^{2}_{z_{i}}-(\partial^{2}_{z_{i}}\omega)\partial_{z_{n}}-2(\partial_{z_{i}}\omega)\partial_{z_{i}}\partial_{z_{n}}+(\partial_{z_{i}}\omega)^{2}\partial^{2}_{z_{n}} & \mathrm{if} \ i=1, \cdots, n-1, \\
\partial^{2}_{z_{n}} & \mathrm{if} \ i=n,
\end{cases}
\end{equation*}
\begin{equation*}
-\Delta_{y}=-\Delta_{z}+(\Delta'_{z'}\omega)\partial_{z_{n}}+2\nabla'_{z'}\omega\cdot\nabla'_{z'}\partial_{z_{n}}-|\nabla'_{z'}\omega|^{2}\partial^{2}_{z_{n}}
\end{equation*}
that $w$ satisfies the following problem in $\mathbb{R}^{n}_{+}\times\mathbb{R}_{+}$:
\begin{equation}
\begin{split}
&\partial_{t}w-\kappa\Delta w+w=h_{\delta}+r_{\delta} & \mathrm{in} \ \mathbb{R}^{n}_{+}\times\mathbb{R}_{+}, \\
&w|_{t=0}=0 & \mathrm{in} \ \mathbb{R}^{n}_{+}, \\
&\kappa\partial_{z_{n}}w|_{\partial\mathbb{R}^{n}_{+}}=h_{b,\delta} & \mathrm{on} \ \partial\mathbb{R}^{n}_{+}\times\mathbb{R}_{+},
\end{split}
\end{equation}
where $r_{\delta}=\kappa\{-(\Delta'\omega)\partial_{z_{n}}w-2\nabla'\omega\cdot\nabla'(\partial_{z_{n}}w)+|\nabla'\omega|^{2}\partial^{2}_{z_{n}}w\}$, $h_{b,\delta}=\kappa(\partial_{z_{n}}\varphi)u+\{\kappa(\nabla'\omega\cdot\nabla'\varphi)u\}(1+|\nabla'\omega|^{2})^{-1}+\kappa_{s}w(1+|\nabla'\omega|^{2})^{-1/2}$.
Moreover, it can be easily seen from $f(t)=0$, $u(t)=0$ for any $t<0$ that $h_{\delta}, r_{\delta} \in L_{q,0}(\mathbb{R}_{+};L_{p}(\mathbb{R}^{n}_{+}))$, $h_{b,\delta} \in H^{1,1/2}_{p,q,0}(\mathbb{R}^{n}_{+}\times\mathbb{R}_{+})$.
By applying Lemma 3.2 with $\lambda_{1}=0$ to (3.52), we obtain that
\begin{equation}
\|w\|_{W^{2,1}_{p,q}(\mathbb{R}^{n}_{+}\times\mathbb{R}_{+})}\leq C(\|h_{\delta}\|_{L_{q}(\mathbb{R}_{+};L_{p}(\mathbb{R}^{n}_{+}))}+\|r_{\delta}\|_{L_{q}(\mathbb{R}_{+};L_{p}(\mathbb{R}^{n}_{+}))}+\|h_{b,\delta}\|_{H^{1,1/2}_{p,q,0}(\mathbb{R}^{n}_{+}\times\mathbb{R}_{+})}),
\end{equation}
where $C$ is a positive constant depending only on $n$, $p$, $q$ and $\kappa$.
It is clear from the definition of $h_{\delta}$, $r_{\delta}$ and $h_{b,\delta}$ that
\begin{equation}
\|h_{\delta}\|_{L_{q}(\mathbb{R}_{+};L_{p}(\mathbb{R}^{n}_{+}))}\leq \|f\|_{L_{q}(\mathbb{R}_{+};L_{p}(\Omega))}+C_{x_{0},\delta}\|u\|_{L_{q}(\mathbb{R}_{+};W^{1}_{p}(\Omega))},
\end{equation}
where $C_{x_{0},\delta}$ is a positive constant depending only on $n$, $\Omega$, $p$, $q$, $\kappa$, $x_{0}$ and $\delta$,
\begin{equation}
\|r_{\delta}\|_{L_{q}(\mathbb{R}_{+};L_{p}(\mathbb{R}^{n}_{+}))}\leq C_{x_{0},\delta}M_{2}\|u\|_{L_{q}(\mathbb{R}_{+};W^{1}_{p}(\Omega))}+CM_{1}(1+M_{1})\|w\|_{L_{q}(\mathbb{R}_{+};W^{2}_{p}(\mathbb{R}^{n}_{+}))},
\end{equation}
where $C_{x_{0},\delta}$ is a positive constant depending only on $n$, $\Omega$, $p$, $q$, $\kappa$, $x_{0}$ and $\delta$, $C$ is a positive constant depending only on $n$, $p$, $q$ and $\kappa$,
\begin{equation}
\|h_{b,\delta}\|_{H^{1,1/2}_{p,q,0}(\mathbb{R}^{n}_{+}\times\mathbb{R}_{+})}\leq C_{x_{0},\delta}\{(1+M_{1}+M_{2})\|u\|_{L_{q}(\mathbb{R}_{+};W^{1}_{p}(\Omega))}+(1+M_{1})\|\langle \partial_{t} \rangle^{1/2}u\|_{L_{q}(\mathbb{R}_{+};L_{p}(\Omega))}\},
\end{equation}
where $C_{x_{0},\delta}$ is a positive constant depending only on $n$, $\Omega$, $p$, $q$, $\kappa$, $\kappa_{s}$, $x_{0}$ and $\delta$.
We can utilize (3.53)--(3.56) to obtain that
\begin{equation}
\begin{split}
&\{1-CM_{1}(1+M_{1})\}\|w\|_{W^{2,1}_{p,q}(\mathbb{R}^{n}_{+}\times\mathbb{R}_{+})} \\
&\leq C\|f\|_{L_{q}(\mathbb{R}_{+};L_{p}(\Omega))}+C_{x_{0},\delta}\{(1+M_{1}+M_{2})\|u\|_{L_{q}(\mathbb{R}_{+};W^{1}_{p}(\Omega))}+(1+M_{1})\|\langle \partial_{t} \rangle^{1/2}u\|_{L_{q}(\mathbb{R}_{+};L_{p}(\Omega))}\},
\end{split}
\end{equation}
where $C$ is a positive constant depending only on $n$, $\Omega$, $p$, $q$ and $\kappa$, $C_{x_{0},\delta}$ is a positive constant depending only on $n$, $\Omega$, $p$, $q$, $\kappa$, $\kappa_{s}$, $x_{0}$ and $\delta$.
Let $M_{1}(1+M_{1})<1/C$, which is possible in the case where $\varepsilon_{0}$ is sufficiently small due to (3.51).
Then
\begin{equation}
\|w\|_{W^{2,1}_{p,q}(\mathbb{R}^{n}_{+}\times\mathbb{R}_{+})}\leq C_{x_{0},\delta}(\|f\|_{L_{q}(\mathbb{R}_{+};L_{p}(\Omega))}+\|u\|_{L_{q}(\mathbb{R}_{+};W^{1}_{p}(\Omega))}+\|\langle \partial_{t} \rangle^{1/2}u\|_{L_{q}(\mathbb{R}_{+};L_{p}(\Omega))}),
\end{equation}
where $C_{x_{0},\delta}$ is a positive constant depending only on $n$, $\Omega$, $p$, $q$, $\kappa$, $\kappa_{s}$, $x_{0}$, $\delta$, $M_{1}$ and $M_{2}$.
Since $v(y,t)=w(z,t)$, it is obtained from the change of variables $(z',z_{n})=(y',y_{n}-\omega(y'))$ that
\begin{equation*}
\begin{split}
\|v\|_{W^{2,1}_{p,q}(H^{n}_{\omega}\times\mathbb{R}_{+})}\leq& 2\|w\|_{L_{q}(\mathbb{R}_{+};L_{p}(\mathbb{R}^{n}_{+}))}+C(1+M_{1}+M_{2})\|\nabla w\|_{L_{q}(\mathbb{R}_{+};L_{p}(\mathbb{R}^{n}_{+}))} \\
&+C(1+M_{1}+M^{2}_{1})\|\nabla^{2}w\|_{L_{q}(\mathbb{R}_{+};L_{p}(\mathbb{R}^{n}_{+}))}+\|\partial_{t}w\|_{L_{q}(\mathbb{R}_{+};L_{p}(\mathbb{R}^{n}_{+}))},
\end{split}
\end{equation*}
\begin{equation}
\|v\|_{W^{2,1}_{p,q}(H^{n}_{\omega}\times\mathbb{R}_{+})}\leq C(1+M_{1}+M^{2}_{1}+M_{2})\|w\|_{W^{2,1}_{p,q}(\mathbb{R}^{n}_{+}\times\mathbb{R}_{+})},
\end{equation}
where $C$ is a positive constant depending only on $n$, $p$ and $q$.
It follows from (3.58), (3.59) that
\begin{equation}
\|v\|_{W^{2,1}_{p,q}(H^{n}_{\omega}\times\mathbb{R}_{+})}\leq C_{x_{0},\delta}(\|f\|_{L_{q}(\mathbb{R}_{+};L_{p}(\Omega))}+\|u\|_{L_{q}(\mathbb{R}_{+};W^{1}_{p}(\Omega))}+\|\langle \partial_{t} \rangle^{1/2}u\|_{L_{q}(\mathbb{R}_{+};L_{p}(\Omega))}),
\end{equation}
where $C_{x_{0},\delta}$ is a positive constant depending only on $n$, $\Omega$, $p$, $q$, $\kappa$, $\kappa_{s}$, $x_{0}$, $\delta$, $M_{1}$ and $M_{2}$.
Therefore, (3.60) and $v=\varphi u$ imply that
\begin{equation}
\|u\|_{W^{2,1}_{p,q}((B_{\delta}(x_{0})\cap\Omega)\times\mathbb{R}_{+})}\leq C_{x_{0},\delta}(\|f\|_{L_{q}(\mathbb{R}_{+};L_{p}(\Omega))}+\|u\|_{L_{q}(\mathbb{R}_{+};W^{1}_{p}(\Omega))}+\|\langle \partial_{t} \rangle^{1/2}u\|_{L_{q}(\mathbb{R}_{+};L_{p}(\Omega))}),
\end{equation}
where $C_{x_{0},\delta}$ is a positive constant depending only on $n$, $\Omega$, $p$, $q$, $\kappa$, $\kappa_{s}$, $x_{0}$ and $\delta$.
Set $\Omega^{\delta}=\{x \in \Omega \ ; \ \mathrm{dist}(x,\partial\Omega)<\delta\}$.
Then, since $\partial\Omega$ is compact in $\mathbb{R}^{n}$, it follows from (3.61) that
\begin{equation}
\|u\|_{W^{2,1}_{p,q}(\Omega^{\delta_{0}}\times\mathbb{R}_{+})}\leq C(\|f\|_{L_{q}(\mathbb{R}_{+};L_{p}(\Omega))}+\|u\|_{L_{q}(\mathbb{R}_{+};W^{1}_{p}(\Omega))}+\|\langle \partial_{t} \rangle^{1/2}u\|_{L_{q}(\mathbb{R}_{+};L_{p}(\Omega))}),
\end{equation}
where $C$ is a positive constant depending only on $n$, $\Omega$, $p$, $q$, $\kappa_{s}$ and $\kappa$.
It is obvious from (3.46) with $\delta=\delta_{0}/2$ and (3.62) that
\begin{equation}
\|u\|_{W^{2,1}_{p,q}(\Omega\times\mathbb{R}_{+})}\leq C(\|f\|_{L_{q}(\mathbb{R}_{+};L_{p}(\Omega))}+\|u\|_{L_{q}(\mathbb{R}_{+};W^{1}_{p}(\Omega))}+\|\langle \partial_{t} \rangle^{1/2}u\|_{L_{q}(\mathbb{R}_{+};L_{p}(\Omega))}),
\end{equation}
where $C$ is a positive constant depending only on $n$, $\Omega$, $p$, $q$, $\kappa$ and $\kappa_{s}$.

It follows from (3.63) and Lemma 2.7 with $s=1/2$ that
\begin{equation*}
\|u\|_{W^{2,1}_{p,q}(\Omega\times\mathbb{R}_{+})}\leq C\{\|f\|_{L_{q}(\mathbb{R}_{+};L_{p}(\Omega))}+(1+r^{-1/2})\|u\|_{L_{q}(\mathbb{R}_{+};W^{1}_{p}(\Omega))}+r^{1/2}\|\partial_{t}u\|_{L_{q}(\mathbb{R}_{+};L_{p}(\Omega))}\},
\end{equation*}
\begin{equation}
(1-Cr^{1/2})\|u\|_{W^{2,1}_{p,q}(\Omega\times\mathbb{R}_{+})}\leq C\{\|f\|_{L_{q}(\mathbb{R}_{+};L_{p}(\Omega))}+(1+r^{-1/2})\|u\|_{L_{q}(\mathbb{R}_{+};W^{1}_{p}(\Omega))}\},
\end{equation}
where $C$ is a positive constant depending only on $n$, $\Omega$, $p$, $q$, $\kappa$ and $\kappa_{s}$.
Let $r<(1/C)^{2}$.
Then it is clear that (3.43) is established by (3.64).
\end{proof}
Second, we discuss (1.1) with a positive constant $\kappa$ and $u_{0}=0$, that is, the following initial-boundary value problem in $\Omega\times\mathbb{R}_{+}$:
\begin{equation}
\begin{split}
&\partial_{t}u-\kappa\Delta u=f & \mathrm{in} \ \Omega\times\mathbb{R}_{+}, \\
&u|_{t=0}=0 & \mathrm{in} \ \Omega, \\
&\kappa\partial_{\nu}u+\kappa_{s}u|_{\partial\Omega}=g & \mathrm{on} \ \partial\Omega\times\mathbb{R}_{+}.
\end{split}
\end{equation}
Theorem 3.1 and Lemma 3.3 allow us to obtain the following lemma:
\begin{lemma}
Let $\kappa$ be a positive constant, $\kappa_{s} \in C^{1}(\partial\Omega)$ satisfy $\kappa_{s}>0$ on $\partial\Omega$, $1<p<\infty$, $1<q<\infty$, $f \in L_{q,0}(\mathbb{R}_{+};L_{p}(\Omega))$, $g \in H^{1,1/2}_{p,q,0}(\Omega\times\mathbb{R}_{+})$.
Then $(3.65)$ has uniquely a solution $u \in W^{2,1}_{p,q,0}(\Omega\times\mathbb{R}_{+})$ satisfying
\begin{equation}
\|u\|_{W^{2,1}_{p,q}(\Omega\times\mathbb{R}_{+})}\leq C(\|f\|_{L_{q}(\mathbb{R}_{+};L_{p}(\Omega))}+\|g\|_{H^{1,1/2}_{p,q,0}(\Omega\times\mathbb{R}_{+})}),
\end{equation}
where $C$ is a positive constant depending only on $n$, $\Omega$, $p$, $q$, $\kappa$ and $\kappa_{s}$.
\end{lemma}
\begin{proof}
Since $C^{\infty}_{0}(\mathbb{R}_{+};X)$ and $C^{\infty}_{0}(\mathbb{R}_{+};W^{1}_{p}(\Omega))$ are dense in $L_{q,0}(\mathbb{R}_{+};X)$ for any Banach space $(X,\|\cdot\|_{X})$ and in $H^{1,1/2}_{p,q,0}(\Omega\times\mathbb{R}_{+})$ respectively, we can assume that $f \in C^{\infty}_{0}(\mathbb{R}_{+};L_{p}(\Omega))$, $g \in C^{\infty}_{0}(\mathbb{R}_{+};W^{1}_{p}(\Omega))$.
It follows from Theorem 3.1 that
\begin{equation*}
\begin{split}
&-\kappa\Delta v=f & \mathrm{in} \ \Omega\times\mathbb{R}_{+}, \\
&\kappa\partial_{\nu}v+\kappa_{s}v|_{\partial\Omega}=g & \mathrm{on} \ \partial\Omega\times\mathbb{R}_{+}
\end{split}
\end{equation*}
has uniquely a solution $v \in C^{\infty}_{0}(\mathbb{R}_{+};W^{2}_{p}(\Omega))$.
Moreover, it is obtained from Lemma 3.3 that
\begin{equation*}
\begin{split}
&\partial_{t}w-\kappa\Delta w=-\partial_{t}v & \mathrm{in} \ \Omega\times\mathbb{R}_{+}, \\
&w|_{t=0}=0 & \mathrm{in} \ \Omega, \\
&\kappa\partial_{\nu}w+\kappa_{s}w|_{\partial\Omega}=0 & \mathrm{on} \ \partial\Omega\times\mathbb{R}_{+}
\end{split}
\end{equation*}
has uniquely a solution $w \in W^{2,1}_{p,q,0}(\Omega\times\mathbb{R}_{+})$.
Therefore, it is sufficient for Lemma 3.4 to be proved that $u=v+w$ satisfies (3.65).
Let $\varphi \in C^{\infty}_{0}(\Omega\times\mathbb{R}_{+})$, $1<p^{*}<\infty$ and $1<q^{*}<\infty$ be dual exponents to $p$ and $q$ respectively.
Then there exists a positive constant $T$ such that $\varphi(t)=0$ for any $t>T/2$.
Set $\psi(x,t)=-\varphi(x,T-t)$.
Then it follows from Lemma 3.3 that
\begin{equation}
\begin{split}
&\partial_{t}\Psi-\kappa\Delta \Psi=\psi & \mathrm{in} \ \Omega\times\mathbb{R}_{+}, \\
&\Psi|_{t=0}=0 & \mathrm{in} \ \Omega, \\
&\kappa\partial_{\nu}\Psi+\kappa_{s}\Psi|_{\partial\Omega}=0 & \mathrm{on} \ \partial\Omega\times\mathbb{R}_{+}
\end{split}
\end{equation}
has uniquely a solution $\Psi \in W^{2,1}_{p^{*},q^{*},0}(\Omega\times\mathbb{R}_{+})$ satisfying
\begin{equation}
\|\Psi\|_{W^{2,1}_{p^{*},q^{*}}(\Omega\times\mathbb{R}_{+})}\leq C\|\psi\|_{L_{q^{*}}(\mathbb{R}_{+};L_{p^{*}}(\Omega))},
\end{equation}
where $C$ is a positive constant depending only on $n$, $\Omega$, $p$, $q$, $\kappa$ and $\kappa_{s}$.
Set $\Phi(x,t)=\Psi(x,T-t)$.
Then $\Phi \in W^{2,1}_{p^{*},q^{*},0}(\Omega\times\mathbb{R}_{+})$ is a unique solution to the following problem in $\Omega\times\mathbb{R}_{+}$:
\begin{equation}
\begin{split}
&\partial_{t}\Phi+\kappa\Delta \Phi=\varphi & \mathrm{in} \ \Omega\times\mathbb{R}_{+}, \\
&\Phi|_{t=0}=0 & \mathrm{in} \ \Omega, \\
&\kappa\partial_{\nu}\Phi+\kappa_{s}\Phi|_{\partial\Omega}=0 & \mathrm{on} \ \partial\Omega\times\mathbb{R}_{+}
\end{split}
\end{equation}
satisfying
\begin{equation}
\|\Phi\|_{W^{2,1}_{p^{*},q^{*}}(\Omega\times\mathbb{R}_{+})}\leq C\|\varphi\|_{L_{q^{*}}(\mathbb{R}_{+};L_{p^{*}}(\Omega))},
\end{equation}
where $C$ is a positive constant depending only on $n$, $\Omega$, $p$, $q$, $\kappa$ and $\kappa_{s}$.
Set
\begin{equation*}
(u,v)_{\Omega\times\mathbb{R}_{+}}=\int_{\Omega\times\mathbb{R}_{+}}u(x,t)v(x,t)dxdt,
\end{equation*}
\begin{equation*}
(u,v)_{\partial\Omega\times\mathbb{R}_{+}}=\int_{\partial\Omega\times\mathbb{R}_{+}}u(x,t)v(x,t)ds_{x}dt,
\end{equation*}
and let $\tilde{\nu} \in C^{2}(\overline{\Omega})$ be the extension of $\nu$ to $\overline{\Omega}$.
Then it is derived from integration by parts and (3.69) that
\begin{equation*}
\begin{split}
(u,\varphi)_{\Omega\times\mathbb{R}_{+}}&=(u,\partial_{t}\Phi+\kappa\Delta \Phi)_{\Omega\times\mathbb{R}_{+}} \\
&=-(\partial_{t}u,\Phi)_{\Omega\times\mathbb{R}_{+}}+(\kappa\Delta u,\Phi)_{\Omega\times\mathbb{R}_{+}}+(u,\kappa\partial_{\nu}\Phi)_{\partial\Omega\times\mathbb{R}_{+}}-(\kappa\partial_{\nu}u,\Phi)_{\partial\Omega\times\mathbb{R}_{+}} \\
&=(-\partial_{t}u+\kappa\Delta u,\Phi)_{\Omega\times\mathbb{R}_{+}}-(\kappa_{s}u+\kappa\partial_{\nu}u,\Phi)_{\partial\Omega\times\mathbb{R}_{+}} \\
&=-(f,\Phi)_{\Omega\times\mathbb{R}_{+}}-(g,\Phi)_{\partial\Omega\times\mathbb{R}_{+}},
\end{split}
\end{equation*}
\begin{equation*}
(u,\varphi)_{\Omega\times\mathbb{R}_{+}}=-(f,\Phi)_{\Omega\times\mathbb{R}_{+}}-\sum^{n}_{i=1}(\partial_{x_{i}}(\tilde{\nu}_{i}g),\Phi)_{\Omega\times\mathbb{R}_{+}}-\sum^{n}_{i=1}(\tilde{\nu}_{i}g,\partial_{x_{i}}\Phi)_{\Omega\times\mathbb{R}_{+}}.
\end{equation*}
Moreover, it follows from (3.70) that
\begin{equation*}
\begin{split}
|(u,\varphi)_{\Omega\times\mathbb{R}_{+}}|&\leq C(\|f\|_{L_{q}(\mathbb{R}_{+};L_{p}(\Omega))}+\|g\|_{L_{q}(\mathbb{R}_{+};W^{1}_{p}(\Omega))})\|\Phi\|_{L_{q^{*}}(\mathbb{R}_{+};W^{1}_{p^{*}}(\Omega))} \\
&\leq C(\|f\|_{L_{q}(\mathbb{R}_{+};L_{p}(\Omega))}+\|g\|_{L_{q}(\mathbb{R}_{+};W^{1}_{p}(\Omega))})\|\varphi\|_{L_{q^{*}}(\mathbb{R}_{+};L_{p^{*}}(\Omega))},
\end{split}
\end{equation*}
\begin{equation}
\|u\|_{ L_{q}(\mathbb{R}_{+};L_{p}(\Omega))}\leq C(\|f\|_{L_{q}(\mathbb{R}_{+};L_{p}(\Omega))}+\|g\|_{L_{q}(\mathbb{R}_{+};W^{1}_{p}(\Omega))}),
\end{equation}
where $C$ is a positive constant depending only on $n$, $\Omega$, $p$, $q$, $\kappa$ and $\kappa_{s}$.
Similarly to $(u,\varphi)_{\Omega\times\mathbb{R}_{+}}$, we obtain that
\begin{equation*}
(\partial_{t}u,\varphi)_{\Omega\times\mathbb{R}_{+}}=(f,\partial_{t}\Phi)_{\Omega\times\mathbb{R}_{+}}+\sum^{n}_{i=1}(\partial_{x_{i}}(\tilde{\nu}_{i}g),\partial_{t}\Phi)_{\Omega\times\mathbb{R}_{+}}-\sum^{n}_{i=1}(\tilde{\nu}_{i}\partial_{t}g,\partial_{x_{i}}\Phi)_{\Omega\times\mathbb{R}_{+}},
\end{equation*}
\begin{equation}
|(\partial_{t}u,\varphi)_{\Omega\times\mathbb{R}_{+}}|\leq |(f,\partial_{t}\Phi)_{\Omega\times\mathbb{R}_{+}}|+\sum^{n}_{i=1}|(\partial_{x_{i}}(\tilde{\nu}_{i}g),\partial_{t}\Phi)_{\Omega\times\mathbb{R}_{+}}|+\sum^{n}_{i=1}|(\tilde{\nu}_{i}\partial_{t}g,\partial_{x_{i}}\Phi)_{\Omega\times\mathbb{R}_{+}}|.
\end{equation}
It follows from the Parseval identity that
\begin{equation}
|(f,\partial_{t}\Phi)_{\Omega\times\mathbb{R}_{+}}|\leq \|f\|_{L_{q}(\mathbb{R}_{+};L_{p}(\Omega))}\|\partial_{t}\Phi\|_{L_{q^{*}}(\mathbb{R}_{+};L_{p^{*}}(\Omega))},
\end{equation}
\begin{equation}
|(\partial_{x_{i}}(\tilde{\nu}_{i}g),\partial_{t}\Phi)_{\Omega\times\mathbb{R}_{+}}|\leq C\|g\|_{L_{q}(\mathbb{R}_{+};W^{1}_{p}(\Omega))}\|\partial_{t}\Phi\|_{L_{q^{*}}(\mathbb{R}_{+};L_{p^{*}}(\Omega))}
\end{equation}
for any $i=1, \cdots, n$, where $C$ is a positive constant depending only on $\Omega$,
\begin{equation}
|(\tilde{\nu}_{i}\partial_{t}g,\partial_{x_{i}}\Phi)_{\Omega\times\mathbb{R}_{+}}|\leq \|\langle \partial_{t} \rangle^{1/2}g\|_{L_{q}(\mathbb{R}_{+};L_{p}(\Omega))}\|\langle \partial_{t} \rangle^{1/2}\partial_{x_{i}}\Phi\|_{L_{q^{*}}(\mathbb{R}_{+};L_{p^{*}}(\Omega))}.
\end{equation}
It is clear from (3.72)--(3.75) and Lemma 2.8 that
\begin{equation*}
\begin{split}
|(\partial_{t}u,\varphi)_{\Omega\times\mathbb{R}_{+}}|&\leq C(\|f\|_{L_{q}(\mathbb{R}_{+};L_{p}(\Omega))}+\|g\|_{H^{1,1/2}_{p,q,0}(\Omega\times\mathbb{R}_{+})})\|\Phi\|_{W^{2,1}_{p^{*},q^{*}}(\Omega\times\mathbb{R}_{+})} \\
&\leq C(\|f\|_{L_{q}(\mathbb{R}_{+};L_{p}(\Omega))}+\|g\|_{H^{1,1/2}_{p,q,0}(\Omega\times\mathbb{R}_{+})})\|\varphi\|_{L_{q^{*}}(\mathbb{R}_{+};L_{p^{*}}(\Omega))},
\end{split}
\end{equation*}
\begin{equation}
\|\partial_{t}u\|_{ L_{q}(\mathbb{R}_{+};L_{p}(\Omega))}\leq C(\|f\|_{L_{q}(\mathbb{R}_{+};L_{p}(\Omega))}+\|g\|_{H^{1,1/2}_{p,q,0}(\Omega\times\mathbb{R}_{+})}),
\end{equation}
where $C$ is a positive constant depending only on $n$, $\Omega$, $p$, $q$, $\kappa$ and $\kappa_{s}$.
By applying Theorem 3.1 to the following problem in $\Omega\times\mathbb{R}_{+}$:
\begin{equation*}
\begin{split}
&-\kappa\Delta u=f-\partial_{t}u & \mathrm{in} \ \Omega\times\mathbb{R}_{+}, \\
&\kappa\partial_{\nu}u+\kappa_{s}u|_{\partial\Omega}=g & \mathrm{on} \ \partial\Omega\times\mathbb{R}_{+},
\end{split}
\end{equation*}
we obtain from (3.71), (3.76) that
\begin{equation*}
\|u\|_{L_{q}(\mathbb{R}_{+};W^{2}_{p}(\Omega))}\leq C(\|f\|_{L_{q}(\mathbb{R}_{+};L_{p}(\Omega))}+\|\partial_{t}u\|_{L_{q}(\mathbb{R}_{+};L_{p}(\Omega))}+\|g\|_{L_{q}(\mathbb{R}_{+};W^{1}_{p}(\Omega))}),
\end{equation*}
\begin{equation}
\|u\|_{L_{q}(\mathbb{R}_{+};W^{2}_{p}(\Omega))}\leq C(\|f\|_{L_{q}(\mathbb{R}_{+};L_{p}(\Omega))}+\|g\|_{H^{1,1/2}_{p,q,0}(\Omega\times\mathbb{R}_{+})}),
\end{equation}
where $C$ is a positive constant depending only on $n$, $\Omega$, $p$, $\kappa$ and $\kappa_{s}$.
It is obvious from (3.71), (3.76), (3.77) that $u$ satisfies (3.66).
\end{proof}
Third, we consider (1.1) with $u_{0}=0$ and $g=0$, that is, the following initial-boundary value problem in $\Omega\times\mathbb{R}_{+}$:
\begin{equation}
\begin{split}
&\partial_{t}u-\mathrm{div}(\kappa\nabla u)=f & \mathrm{in} \ \Omega\times\mathbb{R}_{+}, \\
&u|_{t=0}=0 & \mathrm{in} \ \Omega, \\
&\kappa\partial_{\nu}u+\kappa_{s}u|_{\partial\Omega}=0 & \mathrm{on} \ \partial\Omega\times\mathbb{R}_{+}.
\end{split}
\end{equation}
Lemma 3.4 and freezing coefficients play an important role in proving the following lemma:
\begin{lemma}
Let $\kappa \in C^{1}(\overline{\Omega})$ satisfy $\kappa>0$ on $\overline{\Omega}$, $\kappa_{s} \in C^{1}(\partial\Omega)$ satisfy $\kappa_{s}>0$ on $\partial\Omega$, $1<p<\infty$, $1<q<\infty$, $f \in C^{\infty}_{0}(\mathbb{R}_{+};L_{p}(\Omega))$.
Then $(3.78)$ has uniquely a solution $u \in W^{2,1}_{p,q,0}(\Omega\times\mathbb{R}_{+})$ satisfying
\begin{equation*}
u(t)=\int^{t}_{0}e^{-(t-s)A_{p}}f(s)ds,
\end{equation*}
\begin{equation}
\|u\|_{W^{2,1}_{p,q}(\Omega\times\mathbb{R}_{+})}\leq C\|f\|_{L_{q}(\mathbb{R}_{+};L_{p}(\Omega))},
\end{equation}
where $C$ is a positive constant depending only on $n$, $\Omega$, $p$, $q$, $\kappa$ and $\kappa_{s}$.
\end{lemma}
\begin{proof}
It is sufficient for Lemma 3.5 to be proved that $u$ satisfies (3.79).
We can easily see from Lemma 3.3 that
\begin{equation}
\|u\|_{L_{q}(\mathbb{R}_{+};W^{1}_{p}(\Omega))}\leq C\|f\|_{L_{q}(\mathbb{R}_{+};L_{p}(\Omega))},
\end{equation}
where $C$ is a positive constant depending only on $n$, $\Omega$, $p$, $\kappa$ and $\kappa_{s}$.
It follows from (3.71) that we have only to obtain the following inequality:
\begin{equation}
\|u\|_{W^{2,1}_{p,q}(\Omega\times\mathbb{R}_{+})}\leq C(\|f\|_{L_{q}(\mathbb{R}_{+};L_{p}(\Omega))}+\|u\|_{L_{q}(\mathbb{R}_{+};W^{1}_{p}(\Omega))}),
\end{equation}
where $C$ is a positive constant depending only on $n$, $\Omega$, $p$, $q$, $\kappa$ and $\kappa_{s}$.

Since $\kappa \in C^{1}(\overline{\Omega})$, there exists a positive constant $\delta$ for any positive constant $\varepsilon$ such that $|\kappa(x)-\kappa(x_{0})|<\varepsilon$ for any $x, x_{0} \in \overline{\Omega}$, $x \in B_{\delta}(x_{0})$ and such that $\varphi \in C^{\infty}_{0}(\mathbb{R}^{n})$ can be defined as
\begin{equation*}
\begin{cases}
\varphi(x)=1 & \mathrm{if} \ x \in B_{\delta/2}(x_{0}), \\
0<\varphi(x)<1 & \mathrm{if} \ x \in B_{\delta}(x_{0})\setminus\overline{B_{\delta/2}(x_{0})}, \\
\varphi(x)=0 & \mathrm{if} \ x \in \mathbb{R}^{n}\setminus\overline{B_{\delta}(x_{0})}
\end{cases}
\end{equation*}
for any $x_{0} \in \overline{\Omega}$.
It is clear from (3.78) that $\varphi u$ satisfies the following problem in $\Omega\times\mathbb{R}_{+}$:
\begin{equation}
\begin{split}
&\partial_{t}(\varphi u)-\kappa(x_{0})\Delta(\varphi u)=f_{\delta} & \mathrm{in} \ \Omega\times\mathbb{R}_{+}, \\
&(\varphi u)|_{t=0}=0 & \mathrm{in} \ \Omega, \\
&\kappa(x_{0})\partial_{\nu}u+\kappa_{s}u|_{\partial\Omega}=g_{\delta} & \mathrm{on} \ \partial\Omega\times\mathbb{R}_{+}
\end{split}
\end{equation}
for any $x_{0} \in \overline{\Omega}$, where $f_{\delta}=(\varphi f)+(\kappa-\kappa(x_{0}))\Delta(\varphi u)-2\kappa\nabla\varphi\cdot\nabla u+\varphi\nabla\kappa\cdot\nabla u-\kappa(\Delta\varphi)u$, $g_{\delta}=-(\kappa-\kappa(x_{0}))\partial_{\nu}(\varphi u)+\kappa(\partial_{\nu}\varphi)u$.
Let $\nu$ be extended to $\tilde{\nu} \in C^{2}(\overline{\Omega})$.
Then we obtain by applying Lemmas 2.7, 2.8 and 3.4 to (3.82) that
\begin{equation*}
\|\varphi u\|_{W^{2,1}_{p,q}(\Omega\times\mathbb{R}_{+})}\leq C(\|f_{\delta}\|_{L_{q}(\mathbb{R}_{+};L_{p}(\Omega))}+\|g_{\delta}\|_{H^{1,1/2}_{p,q,0}(\Omega\times\mathbb{R}_{+})}),
\end{equation*}
\begin{equation*}
\|f_{\delta}\|_{L_{q}(\mathbb{R}_{+};L_{p}(\Omega))}\leq C(\|f\|_{L_{q}(\mathbb{R}_{+};L_{p}(\Omega))}+\varepsilon\|\varphi u\|_{L_{q}(\mathbb{R}_{+};W^{2}_{p}(\Omega))})+C_{x_{0},\delta}\|u\|_{L_{q}(\mathbb{R}_{+};W^{1}_{p}(\Omega))},
\end{equation*}
\begin{equation*}
\|g_{\delta}\|_{H^{1,1/2}_{p,q,0}(\Omega\times\mathbb{R}_{+})}\leq C\varepsilon\|\varphi u\|_{W^{2,1}_{p,q}(\Omega\times\mathbb{R}_{+})}+C_{x_{0},\delta}(1+\varepsilon^{-1})\|u\|_{L_{q}(\mathbb{R}_{+};W^{1}_{p}(\Omega))},
\end{equation*}
\begin{equation}
(1-C\varepsilon)\|\varphi u\|_{W^{2,1}_{p,q}(\Omega\times\mathbb{R}_{+})}\leq C\|f\|_{L_{q}(\mathbb{R}_{+};L_{p}(\Omega))}+C_{x_{0},\delta}(1+\varepsilon^{-1})\|u\|_{L_{q}(\mathbb{R}_{+};W^{1}_{p}(\Omega))},
\end{equation}
where $C$ is a positive constant depending only on $n$, $\Omega$, $p$, $q$, $\kappa$ and $\kappa_{s}$, $C_{x_{0},\delta}$ is a positive constant depending only on $n$, $\Omega$, $p$, $q$, $\kappa$, $\kappa_{s}$, $x_{0}$ and $\delta$.
Let $\varepsilon<1/C$.
Then it follows from (3.83) that
\begin{equation}
\|u\|_{W^{2,1}_{p,q}((B_{\delta/2}(x_{0})\cap\Omega)\times\mathbb{R}_{+})}\leq C_{x_{0},\delta}(\|f\|_{L_{q}(\mathbb{R}_{+};L_{p}(\Omega))}+\|u\|_{L_{q}(\mathbb{R}_{+};W^{1}_{p}(\Omega))})
\end{equation}
for any $x_{0} \in \overline{\Omega}$, where $C_{x_{0},\delta}$ is a positive constant depending only on $n$, $\Omega$, $p$, $q$, $\kappa$, $\kappa_{s}$, $x_{0}$ and $\delta$.
Since $\overline{\Omega}$ is compact in $\mathbb{R}^{n}$, (3.81) is established by (3.84).
\end{proof}
Finally, we discuss (1.1) with $u_{0}=0$, that is, the following initial-boundary value problem in $\Omega\times\mathbb{R}_{+}$:
\begin{equation}
\begin{split}
&\partial_{t}u-\kappa\Delta u=f & \mathrm{in} \ \Omega\times\mathbb{R}_{+}, \\
&u|_{t=0}=0 & \mathrm{in} \ \Omega, \\
&\kappa\partial_{\nu}u+\kappa_{s}u|_{\partial\Omega}=g & \mathrm{on} \ \partial\Omega\times\mathbb{R}_{+}.
\end{split}
\end{equation}
It can be easily seen from Theorem 3.1 and Lemma 3.5 that we obtain the following $L_{p}$-$L_{q}$ estimates for time-global solutions to (3.85):
\begin{lemma}
Let $\kappa \in C^{1}(\overline{\Omega})$ satisfy $\kappa>0$ on $\overline{\Omega}$, $\kappa_{s} \in C^{1}(\partial\Omega)$ satisfy $\kappa_{s}>0$ on $\partial\Omega$, $1<p<\infty$, $1<q<\infty$, $f \in L_{q,0}(\mathbb{R}_{+};L_{p}(\Omega))$, $g \in H^{1,1/2}_{p,q,0}(\Omega\times\mathbb{R}_{+})$.
Then $(3.85)$ has uniquely a solution $u \in W^{2,1}_{p,q,0}(\Omega\times\mathbb{R}_{+})$ satisfying
\begin{equation}
\|u\|_{W^{2,1}_{p,q}(\Omega\times\mathbb{R}_{+})}\leq C(\|f\|_{L_{q}(\mathbb{R}_{+};L_{p}(\Omega))}+\|g\|_{H^{1,1/2}_{p,q,0}(\Omega\times\mathbb{R}_{+})}),
\end{equation}
where $C$ is a positive constant depending only on $n$, $\Omega$, $p$, $q$, $\kappa$ and $\kappa_{s}$.
\end{lemma}
\begin{proof}
It is the same as in Lemma 3.4.
\end{proof}
\begin{theorem}
Let $\kappa \in C^{1}(\overline{\Omega})$ satisfy $\kappa>0$ on $\overline{\Omega}$, $\kappa_{s} \in C^{1}(\partial\Omega)$ satisfy $\kappa_{s}>0$ on $\partial\Omega$, $1<p<\infty$, $1<q<\infty$.
Then there exists a positive constant $\lambda^{0}_{1}$ depending only on $n$, $\Omega$, $p$, $q$, $\kappa$ and $\kappa_{s}$ such that if $e^{\lambda_{1}t}f \in L_{q,0}(\mathbb{R}_{+};L_{p}(\Omega))$, $e^{\lambda_{1}t}g \in H^{1,1/2}_{p,q,0}(\Omega\times\mathbb{R}_{+})$ for some $0\leq\lambda_{1}\leq\lambda^{0}_{1}$, then $(3.85)$ has uniquely a solution $u \in W^{2,1}_{p,q,0}(\Omega\times\mathbb{R}_{+})$ satisfying
\begin{equation}
\|e^{\lambda_{1}t}u\|_{W^{2,1}_{p,q}(\Omega\times\mathbb{R}_{+})}\leq C_{p,q,\lambda_{1}}(\|e^{\lambda_{1}t}f\|_{L_{q}(\mathbb{R}_{+};L_{p}(\Omega))}+\|e^{\lambda_{1}t}g\|_{H^{1,1/2}_{p,q,0}(\Omega\times\mathbb{R}_{+})}),
\end{equation}
where $C_{p,q,\lambda_{1}}$ is a positive constant depending only on $n$, $\Omega$, $p$, $q$, $\kappa$, $\kappa_{s}$ and $\lambda_{1}$.
\end{theorem}
\begin{proof}
It is clear from the identity $e^{\lambda_{1}t}\partial_{t}u=-\lambda_{1}e^{\lambda_{1}t}u+\partial_{t}(e^{\lambda_{1}t}u)$ that
\begin{equation}
\begin{split}
&\partial_{t}(e^{\lambda_{1}t}u)-\mathrm{div}(\kappa\nabla(e^{\lambda_{1}t}u))=(e^{\lambda_{1}t}f)+\lambda_{1}(e^{\lambda_{1}t}u) & \mathrm{in} \ \Omega\times\mathbb{R}_{+}, \\
&(e^{\lambda_{1}t}u)|_{t=0}=0 & \mathrm{in} \ \Omega, \\
&\kappa\partial_{\nu}(e^{\lambda_{1}t}u)+\kappa_{s}(e^{\lambda_{1}t}u)|_{\partial\Omega}=(e^{\lambda_{1}t}g) & \mathrm{on} \ \partial\Omega\times\mathbb{R}_{+}.
\end{split}
\end{equation}
It can be easily seen from Lemma 3.6 that (3.88) has uniquely a solution $u \in W^{2,1}_{p,q,0}(\Omega\times\mathbb{R}_{+})$ satisfying
\begin{equation}
\|e^{\lambda_{1}t}u\|_{W^{2,1}_{p,q}(\Omega\times\mathbb{R}_{+})}\leq C(\|e^{\lambda_{1}t}f\|_{L_{q}(\mathbb{R}_{+};L_{p}(\Omega))}+\lambda_{1}\|e^{\lambda_{1}t}u\|_{L_{q}(\mathbb{R}_{+};L_{p}(\Omega))}+\|e^{\lambda_{1}t}g\|_{H^{1,1/2}_{p,q,0}(\Omega\times\mathbb{R}_{+})}),
\end{equation}
where $C$ is a positive constant depending only on $n$, $\Omega$, $p$, $q$, $\kappa$ and $\kappa_{s}$.
Let $\lambda^{0}_{1}<1/C$.
Then (3.89) clearly leads to (3.87).
\end{proof}

\section{Proof of Theorems 2.1 and 2.2}
\subsection{Proof of Theorem 2.1}
Set $v(t)=e^{-tA_{p}}u_{0}$.
Then it is clear from Theorem 3.2 that $v$ is a unique solution to (1.1) with $f=0$, $g=0$ satisfying
\begin{equation}
\|v\|_{W^{2,1}_{p,q}(\Omega\times\mathbb{R}_{+})}\leq C\|u_{0}\|_{X_{p,q}(\Omega)},
\end{equation}
where $C$ is a positive constant depending only on $n$, $\Omega$, $p$, $q$, $\kappa$ and $\kappa_{s}$.
Therefore, it is essential for Theorem 2.1 to be proved that (1.1) has uniquely a solution $u=v+w$, where $w$ is a unique solution to (1.1) with $u_{0}=0$ satisfying
\begin{equation}
\|w\|_{W^{2,1}_{p,q}(\Omega\times\mathbb{R}_{+})}\leq C(\|f\|_{L_{q}((0,T);L_{p}(\Omega))}+\|g\|_{H^{1,1/2}_{p,q,0}(\Omega\times(0,T))}),
\end{equation}
where $C$ is a positive constant depending only on $n$, $\Omega$, $p$, $q$, $\kappa$ and $\kappa_{s}$.
Let $\tilde{f} \in L_{q,0}(\mathbb{R}_{+};L_{p}(\Omega))$ and $\tilde{g} \in H^{1,1/2}_{p,q,0}(\Omega\times\mathbb{R}_{+})$ satisfy $\tilde{f}(t)=f(t)$, $\tilde{g}(t)=g(t)$ for any $0<t<T$, $\|\tilde{f}\|_{L_{q}(\mathbb{R}_{+};L_{p}(\Omega))}=\|f\|_{L_{q}((0,T);L_{p}(\Omega))}$, $\|\tilde{g}\|_{H^{1,1/2}_{p,q,0}(\Omega\times\mathbb{R}_{+})}\leq 2\|g\|_{H^{1,1/2}_{p,q,0}(\Omega\times(0,T))}$.
Then it follows from the definition of $\tilde{f}$ and $\tilde{g}$ that $w$ must satisfy the following problem in $\Omega\times\mathbb{R}_{+}$:
\begin{equation}
\begin{split}
&\partial_{t}w-\mathrm{div}(\kappa\nabla w)=\tilde{f} & \mathrm{in} \ \Omega\times\mathbb{R}_{+}, \\
&w|_{t=0}=0 & \mathrm{in} \ \Omega, \\
&\kappa\partial_{\nu}w+\kappa_{s}w|_{\partial\Omega}=\tilde{g} & \mathrm{on} \ \partial\Omega\times\mathbb{R}_{+}.
\end{split}
\end{equation}
By applying Theorem 3.3 with $\lambda_{1}=0$ to (4.3), we obtain that $w$ is a unique solution to (4.3) satisfying
\begin{equation*}
\|w\|_{W^{2,1}_{p,q}(\Omega\times\mathbb{R}_{+})}\leq C(\|\tilde{f}\|_{L_{q}(\mathbb{R}_{+};L_{p}(\Omega))}+\|\tilde{g}\|_{H^{1,1/2}_{p,q,0}(\Omega\times\mathbb{R}_{+})}),
\end{equation*}
where $C$ is a positive constant depending only on $n$, $\Omega$, $p$, $q$, $\kappa$ and $\kappa_{s}$, which implies (4.2).
Therefore, it is clear that (2.6) is established by (4.1), (4.2).

\subsection{Proof of Theorem 2.2}
Set $v(t)=e^{-tA_{p}}u_{0}$.
Then it follows from Theorem 3.2 that $v$ is a unique solution to (1.1) with $f=0$, $g=0$ satisfying
\begin{equation}
\|e^{\lambda_{1}t}v\|_{W^{2,1}_{p,q}(\Omega\times\mathbb{R}_{+})}\leq C\|u_{0}\|_{X_{p,q}(\Omega)}
\end{equation}
for any $0<\lambda_{1}<\Lambda_{1}/2$, where $C$ is a positive constant depending only on $n$, $\Omega$, $p$, $q$, $\kappa$, $\kappa_{s}$ and $\lambda_{1}$.
Therefore, it is sufficient for Theorem 2.2 that (1.1) has uniquely a solution $u=v+w$, where $w$ satisfies (1.1) with $u_{0}=0$ and
\begin{equation}
\|e^{\lambda_{1}t}w\|_{W^{2,1}_{p,q}(\Omega\times\mathbb{R}_{+})}\leq C(\|e^{\lambda_{1}t}f\|_{L_{q}(\mathbb{R}_{+};L_{p}(\Omega))}+\|e^{\lambda_{1}t}g\|_{H^{1,1/2}_{p,q,0}(\Omega\times\mathbb{R}_{+})})
\end{equation}
for some $0\leq\lambda_{1}\leq\lambda^{0}_{1}$, where $C$ is a positive constant depending only on $n$, $\Omega$, $p$, $q$, $\kappa$, $\kappa_{s}$ and $\lambda_{1}$, $\lambda^{0}_{1}$ is a positive constant defined as in Theorem 3.3.
Let $e^{\lambda_{1}t}\tilde{f} \in L_{q,0}(\mathbb{R}_{+};L_{p}(\Omega))$ satisfy $\tilde{f}(t)=f(t)$ for any $t>0$, $\|e^{\lambda_{1}t}\tilde{f}\|_{L_{q}(\mathbb{R}_{+};L_{p}(\Omega))}=\|e^{\lambda_{1}t}f\|_{L_{q}(\mathbb{R}_{+};L_{p}(\Omega))}$.
Then it follows from the definition of $\tilde{f}$ that $w$ must satisfies the following problem in $\Omega\times\mathbb{R}_{+}$:
\begin{equation}
\begin{split}
&\partial_{t}w-\mathrm{div}(\kappa\nabla w)=\tilde{f} & \mathrm{in} \ \Omega\times\mathbb{R}_{+}, \\
&w|_{t=0}=0 & \mathrm{in} \ \Omega, \\
&\kappa\partial_{\nu}w+\kappa_{s}w|_{\partial\Omega}=\tilde{g} & \mathrm{on} \ \partial\Omega\times\mathbb{R}_{+}.
\end{split}
\end{equation}
It can be easily seen from Theorem 3.3 that (4.6) has uniquely a solution $w$ satisfying
\begin{equation*}
\|e^{\lambda_{1}t}w\|_{W^{2,1}_{p,q}(\Omega\times\mathbb{R}_{+})}\leq C(\|e^{\lambda_{1}t}\tilde{f}\|_{L_{q}(\mathbb{R}_{+};L_{p}(\Omega))}+\|e^{\lambda_{1}t}g\|_{H^{1,1/2}_{p,q,0}(\Omega\times\mathbb{R}_{+})})
\end{equation*}
for some $0\leq\lambda_{1}\leq\lambda^{0}_{1}$, where $C$ is a positive constant depending only on $n$, $\Omega$, $p$, $q$, $\kappa$, $\kappa_{s}$ and $\lambda_{1}$, which implies (4.5).
Therefore, it is obvious that (2.7) is established by (4.4), (4.6).

\section{Global $L_{p}$-$L_{q}$ estimates for solutions to (1.2)}
Let $r$ be, throughout this section, a real number satisfying
\begin{equation*}
\begin{cases}
1<r\leq \displaystyle\frac{n}{n-p} & \mathrm{if} \ 1<p<n, \\
1<r<\infty & \mathrm{if} \ n\leq p<\infty.
\end{cases}
\end{equation*}
It is useful to remark that $L_{p}$-$L_{q}$ estimates for nonlinear terms are obtained from the following lemma:
\begin{lemma}
Let $1<q<\infty$, $I$ be a subinterval on $[0,\infty)$, $(X_{0},\|\cdot\|_{X_{0}})$ and $(X_{1},\|\cdot\|_{X_{1}})$ be a pair of Banach spaces such that $X_{1}$ is dense and continuously included in $X_{0}$.
Then
\begin{equation}
W^{1}_{q}(I;X_{0})\cap L_{q}(I;X_{1})\hookrightarrow BUC(I;(X_{0},X_{1})_{1-1/q,q}),
\end{equation}
where $\hookrightarrow$ is the continuous inclusion.
\end{lemma}
\begin{proof}
It is \cite[Theorem 4.10.2]{Amann}.
\end{proof}
By applying Theorems 2.1 and 2.2 to (1.2), it can be easily seen from the Banach fixed point theorem that $L_{p}$-$L_{q}$ estimates for time-local and time-global solutions to (1.2) are established as follows:
\begin{theorem}
Let $\Omega$ be a bounded domain in $\mathbb{R}^{n}$ with its $C^{1,1}$-boundary $\partial\Omega$, $\kappa \in C^{1}(\overline{\Omega})$ satisfy $\kappa>0$ on $\overline{\Omega}$, $\kappa_{s} \in C^{1}(\partial\Omega)$ satisfy $\kappa_{s}>0$ on $\partial\Omega$, $0<T<\infty$, $1<p<\infty$, $2<q<\infty$, $u_{0} \in X_{p,q}(\Omega)$, $f \in L_{q}((0,T);L_{p}(\Omega))$, $g \in H^{1,1/2}_{p,q,0}(\Omega\times(0,T))$.
Then there exists a positive constant $T_{*}\leq T$ depending only on $n$, $\Omega$, $p$, $q$, $\kappa$, $\kappa_{s}$, $r$, $u_{0}$, $f$ and $g$ such that $(1.2)$ has uniquely a solution $u \in W^{2,1}_{p,q}(\Omega\times(0,T_{*}))$ satisfying
\begin{equation}
\|u\|_{W^{2,1}_{p,q}(\Omega\times(0,T_{*}))}\leq C_{p,q}(\|u_{0}\|_{X_{p,q}(\Omega)}+\|f\|_{L_{q}((0,T);L_{p}(\Omega))}+\|g\|_{H^{1,1/2}_{p,q,0}(\Omega\times(0,T))}),
\end{equation}
where $C_{p,q}$ is a positive constant depending only on $n$, $\Omega$, $p$, $q$, $\kappa$, $\kappa_{s}$ and $r$.
\end{theorem}
\begin{proof}
Set
\begin{equation*}
B(R,T_{*})=\{\bar{u} \in W^{2,1}_{p,q}(\Omega\times(0,T_{*})) \ ; \ \|\bar{u}\|_{W^{2,1}_{p,q}(\Omega\times(0,T_{*}))}\leq R\}
\end{equation*}
for any $R\geq 0$, $0<T_{*}\leq T$.
Let $S$ be the mapping from $B(R,T_{*})$ to $W^{2,1}_{p,q}(\Omega\times(0,T_{*}))$ defined as $S(\bar{u})=u$, where $u$ is a solution to the following problem in $\Omega\times(0,T_{*})$:
\begin{equation}
\begin{split}
&\partial_{t}u-\mathrm{div}(\kappa\nabla u)=f(\bar{u}) & \mathrm{in} \ \Omega\times(0,T_{*}), \\
&u|_{t=0}=u_{0} & \mathrm{in} \ \Omega, \\
&\kappa\partial_{\nu}u+\kappa_{s}u|_{\partial\Omega}=g & \mathrm{on} \ \partial\Omega\times(0,T_{*}),
\end{split}
\end{equation}
where $f(\bar{u})=f+|\bar{u}|^{r-1}\bar{u}$.
Then Theorem 2.1 implies that
\begin{equation}
\|u\|_{W^{2,1}_{p,q}(\Omega\times(0,T_{*}))}\leq C(\|u_{0}\|_{X_{p,q}(\Omega)}+\|f(\bar{u})\|_{L_{q}((0,T_{*});L_{p}(\Omega))}+\|g\|_{H^{1,1/2}_{p,q,0}(\Omega\times(0,T_{*}))}),
\end{equation}
where $C$ is a positive constant depending only on $n$, $\Omega$, $p$, $q$, $\kappa$ and $\kappa_{s}$.
Moreover, it is clear from Lemma 5.1 with $X_{0}=L_{p}(\Omega)$, $X_{1}=W^{2}_{p}(\Omega)$ that
\begin{equation*}
\begin{split}
\|f(\bar{u})\|_{L_{q}((0,T_{*});L_{p}(\Omega))}&\leq \|f\|_{L_{q}((0,T_{*});L_{p}(\Omega))}+\|\bar{u}\|^{r}_{L_{qr}((0,T_{*});L_{pr}(\Omega))} \\
&\leq \|f\|_{L_{q}((0,T);L_{p}(\Omega))}+T^{1/q}_{*}\|\bar{u}\|^{r}_{L_{\infty}((0,T_{*});W^{1}_{p}(\Omega))} \\
&\leq \|f\|_{L_{q}((0,T);L_{p}(\Omega))}+T^{1/q}_{*}\|\bar{u}\|^{r}_{L_{\infty}((0,T_{*});B^{2(1-1/q)}_{p,q}(\Omega))},
\end{split}
\end{equation*}
\begin{equation}
\|f(\bar{u})\|_{L_{q}((0,T_{*});L_{p}(\Omega))}\leq \|f\|_{L_{q}((0,T);L_{p}(\Omega))}+CT^{1/q}_{*}\|\bar{u}\|^{r}_{W^{2,1}_{p,q}(\Omega\times(0,T_{*}))},
\end{equation}
where $C$ is a positive constant depending only on $n$, $\Omega$, $p$, $q$ and $r$.
It can be easily seen from (5.4), (5.5) that
\begin{equation}
\|u\|_{W^{2,1}_{p,q}(\Omega\times(0,T_{*}))}\leq C_{p,q}(\|u_{0}\|_{X_{p,q}(\Omega)}+\|f\|_{L_{q}((0,T);L_{p}(\Omega))}+\|g\|_{H^{1,1/2}_{p,q,0}(\Omega\times(0,T))}+T^{1/q}_{*}R^{r}),
\end{equation}
where $C_{p,q}$ is a positive constant depending only on $n$, $\Omega$, $p$, $q$, $\kappa$, $\kappa_{s}$ and $r$.
Set
\begin{equation}
R=R_{p,q}:=2C_{p,q}(\|u_{0}\|_{X_{p,q}(\Omega)}+\|f\|_{L_{q}((0,T);L_{p}(\Omega))}+\|g\|_{H^{1,1/2}_{p,q,0}(\Omega\times(0,T))}),
\end{equation}
\begin{equation}
T_{1}=\left(\frac{1}{2C_{p,q}R^{r-1}_{p,q}}\right)^{q}.
\end{equation}
Then $S$ can be defined as a mapping in $B(R_{p,q},T_{*})$ for any $0<T_{*}\leq T_{1}$.
Set $u_{i}=S(\bar{u}_{i})$ for any $\bar{u}_{i} \in B(R_{p,q},T_{*})$, $i=1, 2$.
Then it follows from Theorem 2.1, the H\"{o}lder inequality and Lemma 5.1 with $X_{0}=L_{p}(\Omega)$, $X_{1}=W^{2}_{p}(\Omega)$ that
\begin{equation*}
\begin{split}
\|u_{2}-u_{1}\|_{W^{2,1}_{p,q}(\Omega\times(0,T_{*}))}&\leq C\|f(\bar{u}_{2})-f(\bar{u}_{1})\|_{L_{q}((0,T_{*});L_{p}(\Omega))} \\
&\leq C(\|\bar{u}_{1}\|^{r-1}_{L_{qr}((0,T_{*});L_{pr}(\Omega))}+\|\bar{u}_{2}\|^{r-1}_{L_{qr}((0,T_{*});L_{pr}(\Omega))})\|\bar{u}_{2}-\bar{u}_{1}\|_{L_{qr}((0,T_{*});L_{pr}(\Omega))} \\
&\leq CT^{1/q}_{*}(\|\bar{u}_{1}\|^{r-1}_{L_{\infty}((0,T_{*});W^{1}_{p}(\Omega))}+\|\bar{u}_{2}\|^{r-1}_{L_{\infty}((0,T_{*});W^{1}_{p}(\Omega))})\|\bar{u}_{2}-\bar{u}_{1}\|_{L_{\infty}((0,T_{*});W^{1}_{p}(\Omega))},
\end{split}
\end{equation*}
\begin{equation}
\|u_{2}-u_{1}\|_{W^{2,1}_{p,q}(\Omega\times(0,T_{*}))}\leq CR^{r-1}_{p,q}T^{1/q}_{*}\|\bar{u}_{2}-\bar{u}_{1}\|_{W^{2,1}_{p,q}(\Omega\times(0,T_{*}))},
\end{equation}
where $C$ is a positive constant depending only on $n$, $\Omega$, $p$, $q$, $\kappa$, $\kappa_{s}$ and $r$.
Let
\begin{equation}
T_{2}<\left(\frac{1}{CR^{r-1}_{p,q}}\right)^{q}.
\end{equation}
Then $S$ is a contraction mapping in $B(R_{p,q},T_{*})$ for any $0<T_{*}\leq \min\{T_{1},T_{2}\}$.
By applying the Banach fixed point theorem to $B(R_{p,q},T_{*})$ and $S$, where $T_{*}=\min\{T_{1},T_{2}\}$, we can conclude that (1.2) has uniquely a solution in $B(R_{p,q},T_{*})$.
\end{proof}
\begin{theorem}
Let $\Omega$ be a bounded domain in $\mathbb{R}^{n}$ with its $C^{1,1}$-boundary $\partial\Omega$, $\kappa \in C^{1}(\overline{\Omega})$ satisfy $\kappa>0$ on $\overline{\Omega}$, $\kappa_{s} \in C^{1}(\partial\Omega)$ satisfy $\kappa_{s}>0$ on $\partial\Omega$, $1<p<\infty$, $2<q<\infty$, $u_{0} \in X_{p,q}(\Omega)$.
Then there exists a positive constant $\lambda^{0}_{1}$ depending only on $n$, $\Omega$, $p$, $q$, $\kappa$, $\kappa_{s}$ and $r$ such that if $e^{\lambda_{1}t}f \in L_{q}(\mathbb{R}_{+};L_{p}(\Omega))$, $e^{\lambda_{1}t}g \in H^{1,1/2}_{p,q,0}(\Omega\times\mathbb{R}_{+})$ for some $0\leq\lambda_{1}\leq \lambda^{0}_{1}$, then there exists a positive constant $\varepsilon_{\lambda_{1}}$ depending only on $n$, $\Omega$, $p$, $q$, $\kappa$, $\kappa_{s}$, $r$ and $\lambda_{1}$ such that $(1.2)$ has uniquely a solution $u \in W^{2,1}_{p,q}(\Omega\times\mathbb{R}_{+})$ satisfying
\begin{equation}
\|e^{\lambda_{1}t}u\|_{W^{2,1}_{p,q}(\Omega\times\mathbb{R}_{+})}\leq C_{p,q,\lambda_{1}}(\|u_{0}\|_{X_{p,q}(\Omega)}+\|e^{\lambda_{1}t}f\|_{L_{q}(\mathbb{R}_{+};L_{p}(\Omega))}+\|e^{\lambda_{1}t}g\|_{H^{1,1/2}_{p,q,0}(\Omega\times\mathbb{R}_{+})}),
\end{equation}
where $C_{p,q,\lambda_{1}}$ is a positive constant depending only on $n$, $\Omega$, $p$, $q$, $\kappa$, $\kappa_{s}$, $r$ and $\lambda_{1}$ provided that
\begin{equation*}
\|u_{0}\|_{X_{p,q}(\Omega)}+\|e^{\lambda_{1}t}f\|_{L_{q}(\mathbb{R}_{+};L_{p}(\Omega))}+\|e^{\lambda_{1}t}g\|_{H^{1,1/2}_{p,q,0}(\Omega\times\mathbb{R}_{+})}\leq \varepsilon_{\lambda_{1}}.
\end{equation*}
\end{theorem}
\begin{proof}
Set
\begin{equation*}
B(\varepsilon)=\{\bar{u} \in W^{2,1}_{p,q}(\Omega\times\mathbb{R}_{+}) \ ; \ \|e^{\lambda_{1}t}\bar{u}\|_{W^{2,1}_{p,q}(\Omega\times\mathbb{R}_{+})}\leq \varepsilon\}
\end{equation*}
for any $\varepsilon\geq 0$.
Let $S$ be the mapping from $B(\varepsilon)$ to $W^{2,1}_{p,q}(\Omega\times\mathbb{R}_{+})$ defined as $S(\bar{u})=u$, where $u$ is a solution to the following problem in $\Omega\times\mathbb{R}_{+}$:
\begin{equation}
\begin{split}
&\partial_{t}u-\mathrm{div}(\kappa\nabla u)=f(\bar{u}) & \mathrm{in} \ \Omega\times\mathbb{R}_{+}, \\
&u|_{t=0}=u_{0} & \mathrm{in} \ \Omega, \\
&\kappa\partial_{\nu}u+\kappa_{s}u|_{\partial\Omega}=g & \mathrm{on} \ \partial\Omega\times\mathbb{R}_{+},
\end{split}
\end{equation}
where $f(\bar{u})=f+|\bar{u}|^{r-1}\bar{u}$.
Then it follows from Theorem 2.2 that
\begin{equation}
\|e^{\lambda_{1}t}u\|_{W^{2,1}_{p,q}(\Omega\times\mathbb{R}_{+})}\leq C_{\lambda_{1}}(\|u_{0}\|_{X_{p,q}(\Omega)}+\|e^{\lambda_{1}t}f(\bar{u})\|_{L_{q}(\mathbb{R}_{+};L_{p}(\Omega))}+\|e^{\lambda_{1}t}g\|_{H^{1,1/2}_{p,q,0}(\Omega\times\mathbb{R}_{+})}),
\end{equation}
where $C_{\lambda_{1}}$ is a positive constant depending only on $n$, $\Omega$, $p$, $q$, $\kappa$, $\kappa_{s}$ and $\lambda_{1}$, $\lambda^{0}_{1}$ is a positive constant defined as in Theorem 2.2.
Moreover, it is obvious from Lemma 5.1 with $X_{0}=L_{p}(\Omega)$, $X_{1}=W^{2}_{p}(\Omega)$ that
\begin{equation*}
\begin{split}
\|e^{\lambda_{1}t}f(\bar{u})\|_{L_{q}(\mathbb{R}_{+};L_{p}(\Omega))}&\leq \|e^{\lambda_{1}t}f\|_{L_{q}(\mathbb{R}_{+};L_{p}(\Omega))}+\|e^{(\lambda_{1}/r)t}\bar{u}\|^{r-1}_{L_{\infty}(\mathbb{R}_{+};L_{pr}(\Omega))}\|e^{(\lambda_{1}/r)t}\bar{u}\|_{L_{q}(\mathbb{R}_{+};L_{pr}(\Omega))} \\
&\leq \|e^{\lambda_{1}t}f\|_{L_{q}(\mathbb{R}_{+};L_{p}(\Omega))}+\|e^{(\lambda_{1}/r)t}\bar{u}\|^{r-1}_{L_{\infty}(\mathbb{R}_{+};W^{1}_{p}(\Omega))}\|e^{(\lambda_{1}/r)t}\bar{u}\|_{L_{q}(\mathbb{R}_{+};W^{1}_{p}(\Omega))} \\
&\leq \|e^{\lambda_{1}t}f\|_{L_{q}(\mathbb{R}_{+};L_{p}(\Omega))}+\|e^{(\lambda_{1}/r)t}\bar{u}\|^{r-1}_{L_{\infty}(\mathbb{R}_{+};B^{2(1-1/q)}_{p,q}(\Omega))}\|e^{(\lambda_{1}/r)t}\bar{u}\|_{L_{q}(\mathbb{R}_{+};W^{1}_{p}(\Omega))},
\end{split}
\end{equation*}
\begin{equation}
\|e^{\lambda_{1}t}f(\bar{u})\|_{L_{q}(\mathbb{R}_{+};L_{p}(\Omega))}\leq \|e^{\lambda_{1}t}f\|_{L_{q}(\mathbb{R}_{+};L_{p}(\Omega))}+C\|e^{(\lambda_{1}/r)t}\bar{u}\|^{r-1}_{W^{2,1}_{p,q}(\Omega\times\mathbb{R}_{+})}\|e^{(\lambda_{1}/r)t}\bar{u}\|_{L_{q}(\mathbb{R}_{+};W^{1}_{p}(\Omega))},
\end{equation}
where $C$ is a positive constant depending only on $n$, $\Omega$, $p$, $q$ and $r$.
It can be easily seen from (5.13), (5.14) that
\begin{equation}
\|e^{\lambda_{1}t}u\|_{W^{2,1}_{p,q}(\Omega\times\mathbb{R}_{+})}\leq C_{p,q,\lambda_{1}}(\|u_{0}\|_{X_{p,q}(\Omega)}+\|e^{\lambda_{1}t}f\|_{L_{q}(\mathbb{R}_{+};L_{p}(\Omega))}+\|e^{\lambda_{1}t}g\|_{H^{1,1/2}_{p,q,0}(\Omega\times\mathbb{R}_{+})}+\varepsilon^{r}),
\end{equation}
where $C_{p,q,\lambda_{1}}$ is a positive constant depending only on $n$, $\Omega$, $p$, $q$, $\kappa$, $\kappa_{s}$, $r$ and $\lambda_{1}$.
Set
\begin{equation}
\varepsilon=\varepsilon_{p,q,\lambda_{1}}:=2C_{p,q,\lambda_{1}}(\|u_{0}\|_{X_{p,q}(\Omega)}+\|e^{\lambda_{1}t}f\|_{L_{q}(\mathbb{R}_{+};L_{p}(\Omega))}+\|e^{\lambda_{1}t}g\|_{H^{1,1/2}_{p,q,0}(\Omega\times\mathbb{R}_{+})}),
\end{equation}
\begin{equation}
\varepsilon_{1}=\left(\frac{1}{2C_{p,q,\lambda_{1}}}\right)^{1/(r-1)}.
\end{equation}
Then $S$ can be defined as a mapping in $B(\varepsilon_{p,q,\lambda_{1}})$ for any $0\leq \varepsilon_{p,q,\lambda_{1}}\leq \varepsilon_{1}$.
Set $u_{i}=S(\bar{u}_{i})$ for any $\bar{u}_{i} \in B(\varepsilon_{p,q,\lambda_{1}})$, $i=1, 2$.
Then it is derived from Theorem 2.2, the H\"{o}lder inequality and Lemma 5.1 with $X_{0}=L_{p}(\Omega)$, $X_{1}=W^{2}_{p}(\Omega)$ that
\begin{equation*}
\begin{split}
\|e^{\lambda_{1}t}(u_{2}-u_{1})\|_{W^{2,1}_{p,q}(\Omega\times\mathbb{R}_{+})}\leq& C_{\lambda_{1}}\|e^{\lambda_{1}t}(f(\bar{u}_{2})-f(\bar{u}_{1}))\|_{L_{q}(\mathbb{R}_{+};L_{p}(\Omega))} \\
\leq& C_{\lambda_{1}}(\|e^{(\lambda_{1}/r)t}\bar{u}_{1}\|^{r-1}_{L_{\infty}(\mathbb{R}_{+};L_{pr}(\Omega))}+\|e^{(\lambda_{1}/r)t}\bar{u}_{2}\|^{r-1}_{L_{\infty}(\mathbb{R}_{+};L_{pr}(\Omega))}) \\
&\times \|e^{(\lambda_{1}/r)t}(\bar{u}_{2}-\bar{u}_{1})\|_{L_{q}(\mathbb{R}_{+};L_{pr}(\Omega))} \\
\leq& C_{\lambda_{1}}(\|e^{(\lambda_{1}/r)t}\bar{u}_{1}\|^{r-1}_{L_{\infty}(\mathbb{R}_{+};W^{1}_{p}(\Omega))}+\|e^{(\lambda_{1}/r)t}\bar{u}_{2}\|^{r-1}_{L_{\infty}(\mathbb{R}_{+};W^{1}_{p}(\Omega))}) \\
&\times \|e^{(\lambda_{1}/r)t}(\bar{u}_{2}-\bar{u}_{1})\|_{L_{q}(\mathbb{R}_{+};W^{1}_{p}(\Omega))},
\end{split}
\end{equation*}
\begin{equation}
\|e^{\lambda_{1}t}(u_{2}-u_{1})\|_{W^{2,1}_{p,q}(\Omega\times\mathbb{R}_{+})}\leq C_{\lambda_{1}}\varepsilon^{r-1}_{p,q,\lambda_{1}}\|e^{(\lambda_{1}/r)t}(\bar{u}_{2}-\bar{u}_{1})\|_{L_{q}(\mathbb{R}_{+};W^{1}_{p}(\Omega))}
\end{equation}
where $C_{\lambda_{1}}$ is a positive constant depending only on $n$, $\Omega$, $p$, $q$, $\kappa$, $\kappa_{s}$, $r$ and $\lambda_{1}$.
Assume that
\begin{equation}
\varepsilon_{2}<\left(\frac{1}{C_{\lambda_{1}}}\right)^{1/(r-1)}.
\end{equation}
Then $S$ is a contraction mapping in $B(\varepsilon_{p,q,\lambda_{1}})$ for any $0\leq \varepsilon_{p,q,\lambda_{1}}\leq \min\{\varepsilon_{1},\varepsilon_{2}\}$.
By applying the Banach fixed point theorem to $B(\varepsilon_{p,q,\lambda_{1}})$ and $S$ for any $0\leq \varepsilon_{p,q,\lambda_{1}}\leq \min\{\varepsilon_{1},\varepsilon_{2}\}$, we can conclude that (1.2) has uniquely a solution in $B(\varepsilon_{p,q,\lambda_{1}})$.
\end{proof}


\end{document}